\documentclass[12pt,reqno]{amsart}
\usepackage{diagrams}
\usepackage{amssymb}
\usepackage{mathrsfs}
\usepackage{cite}
\usepackage{tikz-cd}
\usepackage{MnSymbol}
\usepackage{a4wide}

\DeclareMathOperator\C{\mathbb C}
\DeclareMathOperator\Z{\mathbb Z}

\newtheorem{theorem}{Theorem}[section]
\newtheorem{lemma}[theorem]{Lemma}
\newtheorem{cor}[theorem]{Corollary}
\newtheorem{prop}[theorem]{Proposition}
\theoremstyle{definition}
\newtheorem{definition}[theorem]{Definition}
\newtheorem{example}[theorem]{Example}

\theoremstyle{remark}

\numberwithin{equation}{section}

\newcommand{\dontprint}[1]\relax

\newcommand{\Lie}{\operatorname{Lie}}
\newcommand{\ind}{\operatorname{ind}}
\newcommand{\Tor}{\operatorname{Tor}}

\newcommand{\coker}{\operatorname{coker}}

\newcommand{\De}{\Delta}

\newcommand{\Aut}{\operatorname{Aut}}
\newcommand{\und}{\underline}
\newcommand{\Pic}{\operatorname{Pic}}
\newcommand{\hra}{\hookrightarrow}

\newcommand{\A}{{\mathbb A}}

\newcommand{\wt}{\widetilde}
\newcommand{\ot}{\otimes}
\newcommand{\Hom}{\operatorname{Hom}}
\newcommand{\Ext}{\operatorname{Ext}}

\newcommand{\Om}{\Omega}
\newcommand{\TT}{{\mathcal T}}

\newcommand{\VV}{{\mathcal V}}
\newcommand{\DD}{{\mathcal D}}

\renewcommand{\SS}{{\mathcal S}}

\newcommand{\HH}{{\mathcal H}}
\newcommand{\GG}{{\mathcal G}}
\newcommand{\LL}{{\mathcal L}}
\newcommand{\MM}{{\mathcal M}}
\newcommand{\OO}{{\mathcal O}}

\newcommand{\UU}{{\mathcal U}}

\newcommand{\si}{\sigma}
\newcommand{\de}{\delta}
\newcommand{\sub}{\subset}

\newcommand{\ov}{\overline}
\newcommand{\im}{\operatorname{im}}
\newcommand{\om}{\omega}
\newcommand{\la}{\lambda}
\renewcommand{\a}{\alpha}

\newcommand{\id}{\operatorname{id}}
\newcommand{\GL}{\operatorname{GL}}

\newcommand{\G}{{\mathbb G}}

\newcommand{\ga}{\gamma}
\newcommand{\lan}{\langle}
\newcommand{\ran}{\rangle}

\renewcommand{\k}{{\mathbf{k}}}

\newcommand{\SL}{{\operatorname{SL}}}
\newcommand{\XX}{{\mathcal X}}

\newcommand{\fg}{{\mathfrak g}}
\newcommand{\fh}{{\mathfrak h}}
\newcommand{\Bun}{{\operatorname{Bun}}}
\newcommand{\End}{{\operatorname{End}}}
\newcommand{\vol}{{\operatorname{vol}}}

\renewcommand{\Re}{{\operatorname{Re}}}
\newcommand{\eps}{\epsilon}
\newcommand{\R}{{\mathbb R}}

\title[Schwartz $\kappa$-densities]
{Schwartz $\kappa$-densities for the moduli stack of rank $2$ bundles on a curve over a local field}

\author{Alexander Braverman}
\author{David Kazhdan}
\author{Alexander Polishchuk}
\thanks{D.K. is partially supported by the ERC grant 101142781.
A.P. is partially supported by the NSF grants DMS-2001224 and DMS-2349388, by the Simons Travel grant MPS-TSM-00002745,
and within the framework of the HSE University Basic Research Program.}

\address{Deparment of Mathematics, University of Toronto, Ontario, Canada}
\email{braval@math.toronto.edu}
\address{Einstein Institute of Mathematics,
The Hebrew University of Jerusalem,
Jerusalem 91904, Israel}
\email{kazhdan@math.huji.ac.il}
\address{
    Department of Mathematics, 
    University of Oregon, 
    Eugene, OR 97403, USA; National Research University Higher School of Economics, Moscow, Russia}
  \email{apolish@uoregon.edu}

\begin{document}

\begin{abstract}
Let $\Bun$ be the moduli stack  of rank $2$ bundles with fixed determinant on a smooth proper irreducible curve $C$ over a local field $F$.
We show how to associate with a Schwartz $\kappa$-density,
for $\Re(\kappa)\ge 1/2$, 
a smooth function on the corresponding coarse moduli space of very stable bundles.
In the non-archimedean case we also prove that the stack $\Bun$ is $\kappa$-bounded in the sense of \cite[Def.\ 2.10]{BK}  for any $\kappa\in\C$.
\end{abstract}

\maketitle

\section{Introduction}

Let $C$ be a smooth (connected) proper curve over a local field $F$ and let $G$ be a reductive group. 

The {\it analytic Langlands correspondence} (that is still only partially understood)
is concerned with the study of certain Hecke operators on Schwartz vector spaces associated with the stack of $G$-bundles over $C$ (see \cite{EFK1, EFK2, EFK3},
\cite{BK} and \cite{BKP}). Conjecturally, these (commuting) operators extend to compact operators on the corresponding Hilbert space with
a discrete joint spectrum, which is expected to be related to some ``Galois side" objects for the Langlands dual group $G^\vee$.

In this paper we study the Schwartz spaces $\SS(\Bun_G,|\om|^\kappa)$ of $\kappa$-densities (where $\kappa\in\C$) on the stack $\Bun_G$ over $F$. In the non-archimedean case
the Schwartz spaces of algebraic stacks were introduced in \cite{GK} (see also \cite{BK} and \cite{BKP}). In the archimedean case similar spaces were introduced in
\cite{Sak}. These Schwartz spaces are defined as inductive limits of the spaces $\SS(X,|p^*\om|^\kappa)_{H(F)}$ of coinvariants, where $[X/H]\hra \Bun_G$ is an open substack
with $X$ a smooth variety, $p:X\to\Bun_G$ is the corresponding map, and $H\simeq \GL_N$ (in the archimedean case taking coinvariants involves taking the closure of
the subspace of elements of the form $hv-v$).

Recall that a $G$-bundle $P$ on $C$ is called {\it very stable} if there are no nonzero nilpotent Higgs fields $\phi\in H^0(C,P_{\fg})$ where
$P_{\fg}$ is the vector bundle corresponding to $P$ and the adjoint representation of $G$. 
Let $M_G^{vs}\sub \Bun_G$ denote the coarse moduli space of very stable bundles. Then \cite[Conjecture 3.5(1)]{BK} states in particular that for $\operatorname{Re}(\kappa)\ge 1/2$, 
every Schwartz $\kappa$-density in $\SS(\Bun_G,|\om|^\kappa)$ defines in a natural way
a smooth $\kappa$-density on $M_G^{vs}(F)$.
The main result of the present paper is a proof of this Conjecture in the case when $G=\SL_2$ in both archimedean and non-archimedean cases. 

Namely, for each line bundle $L_0$ on $C$, we consider the stacks $\Bun_{L_0}$ of rank $2$ bundles $E$ with $\det(E)\simeq L_0$ (for trivial $L_0$ this is exactly
$\Bun_{\SL_2}$), and the open DM-substack of very stable bundles $\MM_{L_0}^{vs}\sub \Bun_{L_0}$, which is dense if $C$ has genus $g\ge 2$. 
The stack $\MM_{L_0}^{vs}$ is a quotient of a smooth variety $M_{L_0}^{vs}$ by the trivial action of $\mu_2$ (so $M_{L_0}^{vs}$ is also the coarse moduli space of $\MM^{vs}_{L_0}$).
We prove
that for $\Re(\kappa)\ge 1/2$, for finite type substacks $[X/G]\sub \Bun_{L_0}$ (with $G=\GL_N$), the integrals of Schwartz $\kappa$-densities on $X(F)$ 
over $F$-points of very stable $G$-orbits are absolutely convergent (away from a subset of orbits of measure zero) and define a map
$$\pi_{\kappa}:\SS(\Bun_{L_0},|\om|^{\kappa})\to C^\infty(M_{L_0}^{vs},|\om|^{\kappa}),$$ 
where in the non-archimedean case the class $C^\infty$ consists of locally constant sections.  
More precisely, we consider push-forwards of Schwartz $\kappa$-densities on $X(F)$ (restricted to the very stable locus) to $M_{L_0}^{vs}$ in the sense of distributions,
and prove that the obtained distributions are of class $C^\infty$ (see Theorem \ref{main-thm-nice}).

In addition, in the non-archimedean case we prove a certain boundedness result (announced in \cite{BK}), Theorem \ref{main-thm-bound}, stating that in the Schwartz spaces 
$\SS(\Bun_{L_0},|\om|^{\kappa})$ one can replace $\Bun_{L_0}$ by a substack of finite type.

Note that in \cite{EFK1} the authors consider in the archimedean case
another Schwartz space of half-densities on $\Bun_G$ defined using global (twisted) differential operators on $\Bun_G$.
Conjecturally their space should coincide with the space $\SS(\Bun_G,|\om|^{1/2})$ used in our paper. This is related to the conjectural statement that the image
of $\pi_{1/2}$ is in $L^2$ (an archimedean analog of \cite[Conj.\ 3.5(3)]{BK}).

The paper is organized as follows. In Section \ref{Schw-sec}, we gather necessary facts about Schwartz spaces. In the archimedean case we refer to
results from \cite{AG}, \cite{AGS} and \cite{AG-sm-tr}. We extend some of these results to the case of Schwartz sections depending holomorphically on a complex
parameter. In Section \ref{coinv-nonarch-sec}, we prove a slight extension of \cite[Thm.\ 6.9]{BZ}, which gives a criterion for an element of the Schwartz coinvariants space
to come from an invariant open subset. 
In Section \ref{coinv-arch-sec} we prove a similar surjectivity criterion in the archimedean case, Theorem \ref{coinv-thm}, which is an extension of \cite[Thm.\ B.0.2]{AG-sm-tr}
and is proved in a similar way.
In Section \ref{orbit-int-sec} we apply the results of Sections \ref{coinv-nonarch-sec} and \ref{coinv-arch-sec} to study the push-forwards of Schwartz $\kappa$-densities on a $G$-variety $X$
to $U_0/G$, where $U_0\sub X$ is an open subset such that $G$-acts freely on $U_0$ and the geometric quotient $U_0/G$ exists. 
The main results are Lemma \ref{extension-lem1} and Lemma \ref{extension-lem2} which compare properties of such push-forwards for Schwartz $\kappa$-densities on $X$ and on
a $G$-invariant open subset $U\sub X$. 
In Section \ref{Bun-sec}, we apply these results to the stacks $\Bun_{L_0}$ and their natural filtration by substacks $\Bun_{L_0}^{\le n}$ (see \ref{filtr-sec}).
The main additional ingredient (which is hard to generalize to the case of higher rank) is finding an explicit form of possible specializations of very stable bundles (see
Sec.\ \ref{vs-lim-sec}). The main results, Theorems \ref{main-thm-nice} and \ref{main-thm-bound} are proved in Sec.\ \ref{main-thm-sec}.

\medskip

{\it Acknowledgment}. We thank Pavel Etingof for help with the proof of Lemma \ref{anal-cont-lem}.

\section{Schwartz spaces}\label{Schw-sec}

\subsection{Schwartz spaces for varieties}\label{Schwartz-var-sec}

Let $F$ be a local field. We denote by $|\cdot |:F^*\to \C^*$ the standard absolute value (so for $F=\C$ we have $|z|=z\ov{z}$, and for a non-archimedean field,
$|a|=q^{-v(a)}$, where $q$ is the number of elements in the residue field).

For any smooth $F$-variety $X$, an (algebraic) line bundle $L$ over $X$, and a continuous character $\chi:F^*\to \C^*$ we have the corresponding
complex line bundle $L^{\chi}$ over $X(F)$. In particular, for  $\kappa\in \C$ we denote  by $|L|^{\kappa}$ the complex line bundle corresponding
to the character $\chi(a)=|a|^{\kappa}$. Note that for a non-archimedean field $F$,
the topological space $X(F)$ is an $l$-space\footnote{This means that it is locally compact, Hausdorff, and every point has a basis of compact open
neighborhoods.}, and the transition functions of $|L|^{\kappa}$ are locally constant.

For example, if $\om_X$ is the algebraic line bundle of top forms on $X$ then $|\om_X|$ is the line bundle of complex densities on $X(F)$.

Assume first that $F$ is non-archimedean. Then we can define Schwartz spaces of line bundles on $X(F)$. More generally, 
if $X$ is an $l$-space and $L$ is a complex line bundle on $X$
(so the transition functions of $L$ are are locally constant), we denote by $C^\infty(X,L)$ the space of locally constant sections,
and define the Schwartz subspace $\SS(X,L)\sub C^\infty(X,L)$ as sections  
with compact support. These definitions also make sense for arbitrary sheaves of $\C$-vector spaces (see \cite[Sec.\ 1.3]{Bern} and Sec.\ \ref{G-eq-sheaves-sec} below).

Now consider the archimedean case, $F=\C$ or $\R$.
Recall that for any smooth variety $X$ over $F$, Aizenbud and Gourevitch define in \cite{AG} the space $\SS(X)$ of Schwartz functions on $X(F)$
(which agrees with the standard Schwartz space in the case $X=\A^n$).
They also define the  space of tempered functions $\TT(X)$, which are smooth functions $f$ such that for any affine open $U\sub X$ and a polynomial differential
operator $D$ on $U$, the function $D(f|_U)$ grows at most polynomially. An important property of tempered functions is that
for $f\in \TT(X)$, one has $f\cdot \SS(X)\sub \SS(X)$.
Note that for any open embedding $U\hra X$ there is a natural linear embedding $\SS(U)\to \SS(X)$.

Now let $V$ be an algebraic vector bundle over $X$, $L_1$ and $L_2$ are algebraic line bundles, and $\kappa\in\C$.
We consider the corresponding complex vector bundle
$$E_{\kappa}:=V\ot_F |L_1|^{\kappa}\ot |L_2|$$
over $X(F)$.
We denote by $C^\infty(X,E_{\kappa})$ the space of smooth sections.

Similarly to \cite{AG} we can define the Schwartz subspace $\SS(X,E_{\kappa})\sub C^\infty(X,E_{\kappa})$.
Note that $E_{\kappa}$ is not a Nash bundle, so the definitions of \cite{AG} do not directly apply. However,
we can still mimic the definition of \cite{AG} by choosing a finite open covering $(U_i)$ such that $V|_{U_i}$, $L_1|_{U_i}$ and $L_2|_{U_i}$
are trivial and considering the linear span of the spaces of $C^\infty$-sections of $E_{\kappa}$ coming from the Schwartz sections of the trivial bundles
$E_{\kappa}|_{U_i}$ over $U_i$ (see Definition \ref{Sch-hol-vec-bun-def} below for a similar definition with holomorphic dependence on $\kappa$).
One can show that the usual properties still hold for the obtained Schwartz spaces $\SS(X,E_{\kappa})$ using tempered partitions of unity
(see \cite[Thm.\ 5.2.1]{AG}) and the fact that the transition functions of $E_{\kappa}$ are tempered.


\subsection{Schwartz sections with holomorphic dependence on a parameter}


Let $X$ be a smooth variety over $F$. We can define Schwartz functions on $X$ depending holomorphically on $\kappa\in \C$.
Let $D\sub \C$ be an open region. 

\begin{definition} We define $\SS^{hol}(X\times D)$ as the space of holomorphic functions on $D$ with values in $\SS(X)$. By definition, these are functions
$f:D\to \SS(X)$, such that for every continuous linear functional (any functional if $F$ is non-archimedean) $\varphi:\SS(X)\to \C$ the function $\varphi\circ f$ is holomorphic.
\end{definition}

Note that if $f$ is a tempered function on $X(F)$ (locally constant in the non-archimedean case), then we have a well defined continuous operator on $\SS^{hol}(X\times D)$ given by 
$\varphi(x,\kappa)\mapsto f(x)\cdot\varphi(x,\kappa)$.

We can also talk about Schwartz sections of the trivial bundle over $X$, depending holomorphically on a parameter $\kappa\in D$
(and this notion does not change under a tempered
change of trivializations).


It is known that in the archimedean case,
\begin{equation}\label{tensor-product-eq}
\SS^{hol}(X\times D)\simeq \SS(X)\hat{\otimes} \HH(D),
\end{equation}
where $\HH(D)$ is the space of holomorphic functions on $D$, and $\hat{\otimes}$ is the topological tensor product of nuclear spaces
(see \cite[ch.II,\S 3, n.3, Ex.2,3]{Groth}).

In the non-archimedean case we have instead
$$\SS^{hol}(X\times D)\simeq \SS(X)\otimes \HH(D).$$


Next, we want to consider Schwartz sections of nontrivial vector bundles. Assume first that $F=\R$ or $\C$.
As before, let $V$ be an algebraic vector bundle over $X$, 
$L_1$ and $L_2$ are algebraic line bundles, and
$$E_{\kappa}:=V\ot_F |L_1|^{\kappa}\ot |L_2|,$$
where $\kappa\in \C$.

Then we can define Schwartz sections of $E_{\kappa}$. 
depending holomorphically on $\kappa\in D$.



\begin{definition}\label{Sch-hol-vec-bun-def} 
We define $\SS^{hol}(X\times D,E_{\kappa})$ as the space of smooth sections of $E_{\kappa}$ over $X\times D$, 
lying in the image of 
$$\bigoplus_i \SS^{hol}(U_i\times D,E_{\kappa}|_{U_i})\to C^\infty(X\times D,E_{\kappa})$$
for some open covering $(U_i)$ of $X$ such that $V|_{U_i}$ and $L_1|_{U_i}$, $L_2|_{U_i}$ are trivial.
\end{definition}

The existence of tempered partition of unity shows that this definition does not depend on a choice of an open covering.

If $F$ is non-archimedean, we can give
an analogous definition for the Schwartz space $\SS^{hol}(X\times D,|L_1|^{\kappa}\ot |L_2|)$ (i.e., for trivial $V$).

Note that if $U\sub X$ is Zariski open then we have a natural embedding (extension by zero),
$$\SS^{hol}(U\times D,E_{\kappa})\to \SS^{hol}(X\times D,E_{\kappa}).$$

Also, for every $\kappa\in D$ we have a continuous map
$$\SS^{hol}(X\times D,E_{\kappa})\to \SS^{hol}(X,E_{\kappa})$$
given by the evaluation at $\kappa$.



\begin{lemma}\label{surj-int-lem}
Let $f:Y\to X$ be a smooth morphism of algebraic varieties, surjective on $F$-points, and let $E'_{\kappa}=f^*E_{\kappa}\ot |\om_f|$ (in the nonarchimedean case
we assume $V$ to be trivial).
Then for each $D$, the integration map
$$f_*:\SS^{hol}(Y\times D,E'_{\kappa})\to \SS^{hol}(X\times D,E_{\kappa})$$
is surjective.
\end{lemma}

\begin{proof}
Let us give a proof for archimedean $F$. The non-archimedean case is analogous but simpler. 
By the partition of unity, the assertion is local in $X$, so we can assume that $E_{\kappa}$ is trivial, and we just need to prove the surjectivity
of the map
$$f_*:\SS^{hol}(Y\times D)\to \SS^{hol}(X\times D).$$
But this follows from the surjectivity of $f_*:\SS(Y)\to \SS(X)$ (see \cite[Thm.\ B.2.4]{AGS}) and from isomorphisms \eqref{tensor-product-eq}
(using the fact that all the spaces involved are nuclear Frechet spaces). 
\end{proof}

\subsection{Schwartz spaces for stacks and coinvariants}\label{schw-stacks-sec}

We consider admissible smooth algebraic stacks over $F$ (see \cite[Sec.\ 2.2]{BKP}).
Recall that they are increasing unions of open substacks of the form $[X/G]$ where $G$ is a linear algebraic group over $F$ acting on a smooth $F$-variety $X$.
The Schwartz spaces can still be defined for line bundles over such stacks. In the non-archimedean case this is done in \cite{GK}
(see also \cite[Sec.\ 2.2]{BKP}). In the archimedean case this is due to Sakellaridis (see \cite{Sak}).

In the case of a quotient stack $[X/G]$, where $G=\GL_N$, the Schwartz spaces $\SS([X/G],|\LL|^{\kappa})$ can be computed as (algebraic) coinvariants of 
$\SS(X,|L|^{\kappa}\ot |{\bigwedge}^{top}(\fg)|^{-1})$ with respect to the $G(F)$-action, where $L$ is the pull-back of $\LL$ to $X$.
In the archimedean case, the definition of coinvariants $\SS(X,|L|^{\kappa}\ot |{\bigwedge}^{top}(\fg)|^{-1})_{G(F)}$
is slightly different: one has to take the quotient by the closure of the span of elements of the form $gv-v$.
Note that we can pull out the factor $|{\bigwedge}^{top}(\fg)|^{-1}$ when taking coinvariants since $G=\GL_N$ acts trivially on it.

Similarly, for an open domain $D\sub \C$, we define $\SS^{hol}([X/G]\times D,|\LL|^{\kappa})$ as appropriate coinvariants
of $\SS(X\times D,|L|^{\kappa}\ot |{\bigwedge}^{top}(\fg)|^{-1})$.

\begin{lemma}\label{Lie-group-coinv-lem}
Assume $G=\GL_N$ and $F$ is archimedean. Then there are natural surjective maps
$$\SS(X,|L|^{\kappa})_{\fg}\to \SS(X,|L|^{\kappa})_{G(F)},$$
$$\SS^{hol}(X\times D,|L|^{\kappa})_{\fg}\to \SS^{hol}(X\times D,|L|^{\kappa})_{G(F)},$$
where we use the usual algebraic coinvariants with respect to the Lie algebra $\fg$.
\end{lemma}

\begin{proof}
This follows from the fact that $\fg$ acts trivially on $\SS(X,|L|^{\kappa})_{G(F)}$ (resp., $\SS^{hol}(X\times D,|L|^{\kappa})_{G(F)}$).
\end{proof}



\subsection{Borel's theorem with a holomorphic parameter}

Here we assume that $F=\R$ or $\C$. 
Let $X$ be a smooth $F$-variety, $Z\sub X$ a smooth closed subvariety, and let $N_Z$ be the normal bundle to $Z$ in $X$.
Let us consider the decreasing filtration on $\SS^{hol}(X\times D,E_{\kappa})$ defined as follows:
$F_Z^i\SS^{hol}(X\times D,E_{\kappa})$
consists of $f(x,\kappa)$ such that for any algebraic differential operator $D$ of order $\le i-1$ defined
on an open $U\sub X$, one has $Df|_{U\cap Z}=0$.

\begin{lemma}\label{filtr-lem}
One has
$$F_Z^i\SS^{hol}(X\times D,E_{\kappa})/F_Z^{i+1}\SS^{hol}(X\times D,E_{\kappa})\simeq \SS^{hol}(Z\times D,S^iN_Z^\vee\ot E|_Z).$$
\end{lemma}

\begin{proof} It is clear that the natural map from the left-hand side to the right-hand side is injective. Thus, we need to prove surjectivity.
Using partition of unity and local \'etale coordinates, we can reduce to the case of a trivial bundle $E$ and of an embedding $Z=\A^k\sub \A^n\times \A^k=X$.
Let us partition coordinates on $X$ as $(y,z)$ where $z$ are coordinates on $Z$. Then we can identify $\SS^{hol}(Z\times D,S^iN_Z^\vee)$ with
functions $f(y,z,\kappa)$ which are homogeneous polynomial of degree $i$ in $y$, with coefficients that are in $\SS^{hol}(Z\times D)$. But we can view such $f$ as elements
of $F_Z^i\SS^{hol}(X\times D,E_{\kappa})$.
\end{proof}

We need the following analog of Borel's theorem.

\begin{prop}\label{Borel-thm-prop}
For each open subset $D\sub \C$ let us consider the map
$$\tau_D:\SS^{hol}(X\times D,E_{\kappa})/\SS^{hol}((X-Z)\times D,E_{\kappa})\to \varprojlim \SS^{hol}(X\times D,E_{\kappa})/F_Z^i\SS^{hol}(X\times D,E_{\kappa}).$$
Then it is injective,
and for any open $D'$ such that $\ov{D'}$ is compact and is contained in $D$, the natural restriction map
\begin{equation}\label{completion-res-map}
\varprojlim \SS^{hol}(X\times D,E_{\kappa})/F_Z^i\SS^{hol}(X\times D,E_{\kappa})\to \varprojlim \SS^{hol}(X\times D',E_{\kappa})/F_Z^i\SS^{hol}(X\times D',E_{\kappa})
\end{equation}
factors through the image of $\tau_{D'}$.
\end{prop}

\begin{proof} 
First, we claim that $\SS^{hol}((X-Z)\times D,E_{\kappa})$ is exactly the kernel of the map
$$\wt{\tau}_D:\SS^{hol}(X\times D,E_{\kappa})\to \varprojlim \SS^{hol}(X\times D,E_{\kappa})/F_Z^i\SS^{hol}(X\times D,E_{\kappa}).$$
Indeed, this immediately follows from the similar statement for a fixed $\kappa$.

Thus, it remains to check that the map \eqref{completion-res-map} factors through the image of $\wt{\tau}_{D'}$.
Using partition of unity and local \'etale coordinates, we reduce to the case of trivial $E$ and of an embedding $Z=\A^k\sub \A^n\times \A^k=X$.
Then the element of $\varprojlim \SS^{hol}(X\times D,E_{\kappa})/F_Z^i\SS^{hol}(X\times D,E_{\kappa})$
can be viewed as a formal series 
$$f(y,z,\kappa)=\sum_{i\ge 0}f_i(y,z,\kappa),$$
where $f_i(y,z,\kappa)$ is a homogeneous polynomial of degree $i$ in $y$, with coefficients that are in $\SS^{hol}(Z\times D)$.
We claim that there exists then $F(y,z,\kappa)\in \SS^{hol}(X\times D')$,
such that the formal expansion of $F$ in $y$ is equal to $f|_{D'}$.
Indeed, this follows from the standard proof of Borel's theorem. 
\end{proof}



\section{Coinvariants for Schwartz sections: non-archimedean case}\label{coinv-nonarch-sec}

\subsection{Homology for smooth representations of $l$-groups}

We use the setup of $l$-spaces and $l$-groups from \cite{BZ} (see also \cite{Bern}).

Let $G$ be an $l$-group, $\HH(G)$ the Hecke algebra of locally constant compactly supported distributions on $G$,
$\MM(G)$ the category of smooth representations of $G$. Recall that it is equivalent to the category $\MM(\HH(G))$ of nondegenerate left $\HH(G)$-modules 
(see \cite[Sec.\ 2.1]{Bern}).
It is known that this category has enough projectives (see \cite[Sec.\ 2.2]{Bern}).

\begin{lemma}\label{Hecke-proj-lem}
$\HH(G)$ is projective as a left $\HH(G)$-module.
\end{lemma}

\begin{proof} $\HH(G)$ is the direct limit of the projective modules $\HH(G)e_K$, where $e_K$ is the idempotent corresponding to a compact open subgroup $K\sub G$.
Hence, for any $M\in \MM(G)$ we have
$$\Hom_{\HH(G)}(\HH(G),M)\simeq \varprojlim_K \Hom(\HH(G)e_K,M)=\varprojlim_K e_KM,$$
where the maps $e_KM\to e_{K'}M$ for $K\sub K'$ are given by the multiplication with $e_{K'}$. But these maps are surjective (since $e_{K'}=e_{K'}e_K$),
so for an exact sequence $0\to M_1\to M_2\to M_3\to 0$ we get an exact sequence of inverse systems satisfying the Mittag-Leffler condition, which implies
that the functor $\Hom_{\HH(G)}(\HH(G),\cdot)$ is exact.
\end{proof}

For a smooth $G$-representation $M$, we denote by $M_G$ the space of $G$-coinvariants.
The integration defines a homomorphism $\HH(G)\to \C$, and we view $\C$ as a right $\HH(G)$-module via this homomorphism.
It is easy to see that for $M\in \MM(G)$ there is a natural isomorphism
\begin{equation}\label{coinv-tensor-eq}
M_G\simeq \C\ot_{\HH(G)} M.
\end{equation}

\begin{definition}
We define smooth homology $H^{sm}_i(G,M)$ of $G$ with coefficients in $M\in \MM(G)$ as the left derived functors of the functor of coinvariants $M\mapsto M_G$
on the category $\MM(G)$.
\end{definition}

It follows from \eqref{coinv-tensor-eq} that 
$$H^{sm}_i(G,M)\simeq \Tor_i^{\HH(G)}(\C,M).$$

\begin{example}
If $G$ is a union of compact subgroups then the functor $M\mapsto M_G$ is exact, so $H^{sm}_{>0}(G,M)=0$.
\end{example}

\begin{lemma}\label{sm-hom-char-lem}
Let $F$ be a local non-archimedean field, $G=F^*\ltimes F^n$, where $F^*$ acts linearly on $F^n$.
Let $\chi:F^*\to \C^*$ be a character, trivial on some open subgroup of $F^*$. We denote by $\C_\chi$ the smooth $1$-dimensional
representation of $G$, where $\chi$ extends to $G$ via the projection $G\to F^*$. Assume $\chi\neq 1$. Then
$$H^{sm}_*(G,\C_\chi)=0.$$
\end{lemma}

\begin{proof}
Since $F^n$ is a union of compact subgroups, this reduces to the case $G=F^*$. Moreover, we can replace $G$ by
its quotient $G_0$ by an open compact subgroup. Note that $G_0$ is an extension of $\Z$ by a finite commutative group $K$,
and we just need to compute the usual homology $H_*(G_0,\C_\chi)$. If $\chi|_K\neq 1$, the assertion follows from the spectral sequence.
In the case $\chi|_K=1$, we reduce to the case $G_0=\Z$. Then we immediately compute the homology using the resolution 
$$0\to \C[t,t^{-1}]\rTo{t-1}\C[t,t^{-1}]\to \C\to 0$$
of the trivial module over $\C[\Z]=\C[t,t^{-1}]$.
\end{proof}

Now let $H\sub G$ be a closed subgroup. We denote by
$$\ind_H^G:\MM(H)\to \MM(G)$$
the functor of compact induction. It is known that it is an exact functor.

Note that $\HH(G)$ has a natural $\HH(G)$-bimodule structure, hence, we can view it
as an $\HH(G)-\HH(H)$-bimodule.

For a locally compact group $H$ we denote by $\De_H$ the {\it modular character} $H\to \R^*_{>0}$ (see \cite[0.5]{Cartier}), \cite[Sec.\ 1.19]{BZ}).
In the case when $H=\und{H}(F)$, where $F$ is a local field, and $\und{H}$ is a linear algebraic group over $F$,
we have an {\it algebraic modular character} $\De^{alg}_H:H\to \G_m$
given by the adoint action of $H$ on ${\bigwedge}^{top}\Lie(H)^*$. Then $\De_H$ is the composition of the $\De^{alg}_H(F):H(F)\to F^*$ with $|\cdot|:F^*\to \R^*_{>0}$.

\begin{lemma}\label{ind-tensor-lem}(\cite[Thm.\ 1.4]{Cartier})
For a closed subgroup $H\sub G$ and $M\in \MM(H)$, one has a natural isomorphism
$$\ind_H^G(M)\simeq \HH(G)\ot_{\HH(H)} (M\cdot (\De_H/\De_G|_H)),$$
\end{lemma}


\begin{lemma}\label{ind-proj-lem}
The functor $\ind_H^G$ takes projectives to projectives.
\end{lemma}

\begin{proof} The right adjoint functor to $\ind_H^G$ is given by
$N\mapsto \Hom_{\HH(G)}(\HH(G),N)$, where the $\HH(H)$-structure is induced by the right action of $\HH(H)$ on $\HH(G)$.
The assertion follows from the fact that the latter functor is exact since $\HH(G)$ is projective as an object of $\MM(G)$ by Lemma \ref{Hecke-proj-lem}.
\end{proof}

\begin{lemma}\label{hom-ind-rep-lem}
For $M\in \MM(H)$, one has a natural isomorphism
$$H^{sm}_i(G,\ind_H^G M)\simeq H^{sm}_i(H,M\cdot (\De_H/\De_G|_H)).$$
\end{lemma}

\begin{proof}
For $i=0$, this follows from \cite[Prop.\ 2.29]{BZ}, or can be deduced from Lemma \ref{ind-tensor-lem} together with \eqref{coinv-tensor-eq}. 
The general case follows from Lemma \ref{ind-proj-lem}.
\end{proof}

\subsection{Equivariant sheaves on $l$-spaces}\label{G-eq-sheaves-sec}

Let $G$ be an $l$-group acting on an $l$-space $X$. Let $\MM(G,X)$ denote the category of $G$-equivariant sheaves on $X$
(see \cite[Sec.\ 1.3]{Bern}). Recall that a $G$-equivariant sheaf $F$ on $X$ is equipped with an isomorphism of sheaves on $G\times X$,
$$\a:p^*F\to a^*F,$$
where $p:G\times X\to X$ is the projection and $a:G\times X\to X$ is the action map, satisfying the natural cocycle condition on $G\times G\times X$.

If $X$ is a point then $\MM(G,pt)$ is naturally equivalent to $\MM(G)$, the category of smooth representations of $G$ (see \cite[Sec.\ 1.3]{Bern}).

If $\pi:X\to Y$ is a continuous map of $l$-spaces, compatible with actions of $G$ on $X$ and $Y$, then we have an exact functor
$$\pi_!:\MM(G,X)\to \MM(G,Y).$$
Furthermore, a $G$-equivariant version of the base change holds, similar to \cite[Sec.\ 1.3, Prop.\ 3]{Bern}.

In the case of the projection to the point $\pi:X\to pt$, we have $\pi_!(F)=\SS(X,F)$, the Schwartz  space of compactly supported sections of $F$, with the natural
(smooth) action of $G$.

If $Y$ has a trivial $G$-action, then by the base change, for any point $y\in Y$ and a $G$-equivariant sheaf $F$ on $X$, we have
an isomorphism in $\MM(G)$,
\begin{equation}\label{push-forward-stalk-eq}
(\pi_!F)_y\simeq \SS(\pi^{-1}(y), F|_{\pi^{-1}(y)}).
\end{equation}

If $H\sub G$ is a closed subgroup, then for $M\in \MM(H)$, 
the compactly induced representation $\ind_H^G M$ can be identified with the Schwartz space of
the corresponding $G$-equivariant sheaf on $G/H$. Hence, we get the following Corollary from Lemma \ref{hom-ind-rep-lem}.

\begin{cor}\label{G/H-cor}
Let $H\sub G$ is a closed subgroup, and let $L$ be a $G$-equivariant complex line bundle on $G/H$.
Then
$$H^{sm}_i(G,\SS(G/H,L))\simeq H^{sm}_i(H,\chi_L\cdot (\De_H/\De_G|_H)),$$
where $\chi_L:H\to \C^*$ is the character given by the action of $H$ on the fiber of $L$ at the point with the stabilizer $H$.
\end{cor}

If $U\sub X$ is an open subset, $Z=X\setminus U$ the complementary closed subset, then for any sheaf $F$ on $X$ we have an exact sequence
$$0\to \SS(U,F)\to \SS(X,F)\to \SS(Z,F|_Z)\to 0$$
(see \cite[Sec.\ 1.3, Prop.\ 4]{Bern}). We also have the following co-Cech complexes associated with open coverings.

\begin{lemma}\label{co-Cech-lem}
For any open covering $(U_i)$ of $X$ and any $G$-equivariant sheaf $F$ on $X$, where each $U_i$ is $G$-invariant, we have an exact sequence
of smooth $G$-representations
$$\ldots \bigoplus_{i,j} \SS(U_{ij},F)\to \bigoplus_i \SS(U_i,F)\to \SS(X,F)\to 0,$$
in which the differentials are defined in a dual way to the Cech complex (see \cite[Sec.\ 21.8]{stacks}).
\end{lemma}

\begin{proof} This follows from the exactness of the functor $\SS(X,\cdot)$ and the exactness of the corresponding sheafified complex,
where we use the direct sums of sheaves of the form $j_!F|_U$, for $j:U\hra X$ with $U=U_{i_1\ldots i_k}$.
\end{proof}

It is well known that sheaves of $\C$-vector spaces on an $l$-space $Y$ correspond to nondegenerate $\SS(Y)$-modules (see \cite[Sec.\ 1.3]{Bern}), while
smooth $G$-representations correspond to nondegenerate $\HH(G)$-modules, where $\HH(G)$ is the Hecke algebra.
We can get a similar description for the category of $G$-equivariant sheaves on $Y$, where $G$ acts trivially on $Y$.

\begin{definition}
We say that an $\HH(G)\ot\SS(Y)$-module $M$ is {\it strongly nondegenerate}
if $\HH(G)M=M$ and $\SS(Y)M=M$.
\end{definition}

\begin{lemma}
If $Y$ has a trivial $G$-action then the functor $F\mapsto \SS(Y,F)$ gives an equivalence
of $\MM(G,Y)$ with the category of strongly nondegenerate $\HH(G)\ot \SS(Y)$-modules $M$. 
Under this equivalence the functor of push-forward to the point $F\mapsto \SS(Y,F)\in \MM(G)$ corresponds to restricting the module structure
to $\HH(G)$. Also, for each point $y\in Y$, the stalk functor $F\mapsto F_y\in \MM(G)$ corresponds to the functor 
$M\mapsto M\ot_{\SS(Y)}\C_y$, where $\C_y$ is the $\SS(Y)$-module corresponding to the homomorphism $ev_y:\SS(Y)\to \C$
of evaluation at $y$. 
\end{lemma}

\begin{proof} 
Let us denote by $\MM'(\HH(G)\ot\SS(Y))$ the category of nondegenerate $\HH(G)\ot\SS(Y)$-modules.
First, we observe that the functor $F\mapsto \SS(Y,F)$ can be enriched to a functor 
$$\MM(G,Y)\to \MM'(\HH(G)\ot\SS(Y)).$$
Indeed, the nondegenerate $\SS(Y)$-module $\SS(Y,F)$ has a $G$-action commuting with the $\SS(Y)$-module structure, and
$\SS(Y,F)$ is a smooth $G$-representation. 
Hence, $\SS(Y,F)$ has a nondegenerate $\HH(G)$-module structure commuting with the $\SS(Y)$-module structure, i.e., it becomes
a strongly nondegenerate $\HH(G)\ot\SS(Y)$-module.

Conversely, starting with $M\in \MM'(\HH(G)\ot\SS(Y))$, we get a sheaf $F$ on $Y$ with $\SS(Y,F)=M$, and we can use the smooth $G$-action on $M$
to get a $G$-equivariant structure on $F$. Namely, we know that $M$ corresponds to a $G$-equivariant sheaf on a point, which is given by a map of $\SS(G)$-modules,
$$\a:\SS(G)\ot M\to \SS(G)\ot M$$ 
satisfying a certain cocycle condition. But the map $\a$ commutes with the $\SS(Y)$-module structures, hence, it corresponds to a map $p^*F\to p^*F$ of sheaves on $G\times X$,
so it gives a $G$-equivariant structure on $F$.
\end{proof}

Assume $Y$ has a trivial $G$-action. For each open compact subgroup $K\sub G$ and each open compact $L\sub Y$,
let $P_{K,L}$ denote the object of $\MM(G,Y)$ corresponding to the strongly nondegenerate $\HH(G)\ot\SS(Y)$-module $\HH(G)e_K\ot \SS(Y)e_L$,
where $e_K\in \HH(G)$ (resp., $e_L\in \SS(Y)$) is the idempotent corresponding to $K$ (resp., $L$).

\begin{lemma}\label{P-KL-lem}
The objects $P_{K,L}\in \MM(G,Y)$ are projective. For every object $F$ of $\MM(G,Y)$ there is a surjective from a direct sum of some $P_{K,L}$ to $F$.
For every point $y\in Y$, the stalk $(P_{K,L})_y$ is a projective $\HH(G)$-module. Also, $\SS(Y,P_{K,L})$ is a projective $\HH(G)$-module.
\end{lemma}

\begin{proof}
It is clear that the $\HH(G)\ot\SS(Y)$-module $\HH(G)e_K\ot \SS(Y)e_L$ is projective since $\Hom(\HH(G)e_K\ot \SS(Y)e_L,M)=(e_K\ot e_L)M$.
Furthermore, since for a strongly nondegenerate $\HH(G)\ot \SS(Y)$-module $M$, every $x\in M$ is contained in $(e_K\ot e_L)M$ for
some $(K,L)$, there is a surjection from a direct sum of some $\HH(G)e_K\ot \SS(Y)e_L$ to $M$.

The stalks of $P_{K,L}$ are given by
$$(P_{K,L})_y\simeq (\HH(G)e_K)\ot(\SS(Y)e_L)\ot_{\SS(M)}\C_y)\simeq \begin{cases} \HH(G)e_K, & y\in L\\ 0, &y\not\in L,\end{cases}$$
so they are projective $\HH(G)$-modules. 

Finally, viewed as an $\HH(G)$-module, $\HH(G)e_K\ot \SS(Y)e_L$ is a direct sum of copies of $\HH(G)e_K$, hence it is projective.
\end{proof}

Still assuming that $Y$ has a trivial $G$-action let us consider the functor of $G$-coinvariants
$$\MM(G,Y)\to \MM(Y): F\mapsto F_G,$$
which corresponds to the functor $M\mapsto \C\ot_{\HH(G)}M$ on strongly nondegenerate $\HH(G)\ot\SS(M)$-modules.
Thus, by definition $\SS(Y,F_G)=\SS(Y,F)_G$.
We denote by $F\mapsto H_i(G,F)$ the left derived functors $\MM(G,Y)\to \MM(Y)$.

\begin{lemma}\label{trivial-G-action-lem} 
Assume $Y$ has a trivial $G$-action. Then for $F\in \MM(G,Y)$ and $y\in Y$, we have
natural isomorphisms
$$H_i(G,F)_y\simeq H_i(G,F_y),$$
$$\SS(Y,H_i(G,F))\simeq H_i(G,\SS(Y,F)).$$
\end{lemma}

\begin{proof} For $i=0$ the first isomorphism is \cite[Prop.\ 2.36]{BZ} and the second holds by definition. 
The general case follows from Lemma \ref{P-KL-lem} by considering projective resolutions 
consisting of direct sums of objects $P_{K,L}$.
\end{proof}

Now assume we have an $l$-space $X$ with an action of an $l$-group $G$ and a continuous map of $l$-spaces
$\pi:X\to Y$, such that $\pi(gx)=\pi(x)$ (so it is $G$-equivariant, where $G$ acts trivially on $Y$).

\begin{lemma}\label{sh-G-hom-lem}
For a $G$-equivariant sheaf $F$ on $X$ and a point $y\in Y$, we have natural isomorphisms
$$H_i(G,\pi_!F)_y\simeq H_i(G,\SS(\pi^{-1}(Y),F|_{\pi^{-1}(y)})).$$
\end{lemma}

\begin{proof}
Set $F':=\pi_!F\in\MM(G,Y)$. By Lemma \ref{trivial-G-action-lem}, we have
$$H_i(G,F')_y\simeq H_i(G,F'_y).$$
It remains to apply the isomorphism \eqref{push-forward-stalk-eq}. 
\end{proof}

Recall that the graph of a $G$-action on $X$ is the image of the map $G\times X\to X\times X: (g,x)\mapsto (x,gx)$.

\begin{definition}\label{loc-sep-def}
An action of an $l$-group $G$ on an $l$-space $X$ is called {\it locally separated} (resp., {\it separated}) if the graph of the action in $X\times X$
is locally closed (resp., closed). By \cite[Lem.\ 6.2]{BZ}, this is equivalent to requiring that the diagonal in $(X/G)\times(X/G)$ is locally closed (resp., closed).
\end{definition}

Note that in \cite{BZ} separated group actions are called regular. By \cite[Lem.\ 6.4]{BZ}, if the action of $G$ on $X$ is separated then the quotient $X/G$ is an $l$-space.

\begin{prop}\label{hom-van-prop} 
Assume that the action of $G$ on $X$ is locally separated, and let $F$ be a $G$-equivariant sheaf on $X$.
Fix $i_0\ge 0$. Assume that for every $G$-orbit $\Om\sub X$, one has
$H_i(G,\SS(\Om,F|_{\Om}))=0$ for $i\le i_0$. Then $H_i(G,\SS(X,F))=0$ for $i\le i_0$.
\end{prop}

\begin{proof} Assume first that the action of $G$ on $X$ is separated, set $Y=X/G$, and let $\pi:X\to Y$ be the natural projection,
so that the fibers of $\pi$ are exactly $G$-orbits.

Set $F':=\pi_!F\in \MM(G,Y)$.
By Lemma \ref{sh-G-hom-lem}, our assumption implies that the sheaf $H_i(G,F')$ has zero stalks, hence
it is zero. Now using Lemma \ref{trivial-G-action-lem}, we get
$$H_i(G,\SS(X,F))\simeq H_i(G,\SS(Y,F'))\simeq \SS(Y,H_i(G,F'))=0.$$

Now consider the general case when the action of $G$ on $X$ is locally separated. Then 
the diagonal in $(X/G)\times(X/G)$ is closed in a neighborhood of every point $(\ov{x},\ov{x})$. Hence, there exists an open neighborhood
$\ov{U}$ of $\ov{x}$ in $X/G$, such that the diagonal is closed in $\ov{U}\times \ov{U}$. The preimage $U\sub X$ of $\ov{U}$ is a $G$-equivariant
open subset such that the action of $G$ on $U$ is separated.
Hence, there exists an open covering $(U_i)$ of $X$ such that $U_i$ are $G$-invariant and the action of $G$ on each $U_i$ is separated.
It follows that the action of $G$ on each $U_{i_1\ldots i_k}=U_{i_1}\cap\ldots U_{i_k}$ is separated, and so by the first part of the argument,
$$H_i(G,\SS(U_{i_1\ldots i_k},F))=0$$
for $i\le i_0$. Now the similar vanishing for $H_i(G,\SS(X,F))$ follows from the long exact sequence of smooth $G$-representations
$$\ldots \bigoplus_{i,j} \SS(U_{ij},F)\to \bigoplus_i \SS(U_i,F)\to \SS(X,F)\to 0$$
(see Lemma \ref{co-Cech-lem}).
\end{proof}

\subsection{Adding a coefficient ring}\label{sh-coef-sec}

Let $X$ be an $l$-space, $G$ an $l$-group.
Instead of considering $G$-equivariant sheaves of $\C$-vector spaces on $X$ as before, we can consider
the category $\MM(G,X;R)$ of $G$-equivariant sheaves of $R$-modules on $X$, where $R$ is a commutative $\C$-algebra.

All the results of Section \ref{G-eq-sheaves-sec} extend naturally to this setup.
For example, in the case when $G$ acts trivially on $X$, the category $\MM(G,X;R)$ is equivalent to the category of strongly nondegenerate
$R\ot\HH(G)\ot\SS(X)$-modules.

\begin{definition}
An {\it $R$-line bundle} over $X$ is a sheaf $\LL$ of $R$-modules such that for some open covering $(U_i)$ we have
an isomorphism of $\LL|_{U_i}$ with the sheaf of $R$-valued locally constant functions.
Thus, such $\LL$ is given by the transition functions $f_{ij}$ which are locally constant $R^*$-valued functions on $U_i\cap U_j$.
For an extension of rings $R\hra R'$ and an $R$-line bundle $\LL$, we have a natural $R'$-line bundle $\LL\ot_R R'$, given by
the same transition functions via the embedding $R\hra R'$.
\end{definition}

Let us describe the main example of an $R$-line bundle we are interested in (for the coefficent ring $R=\C[t,t^{-1}]$). 
Consider the homomorphism
$$F^*\to \C[t,t^{-1}]^*: x\mapsto |x|_t:=t^{-v(x)},$$
where $|x|=q^{-v(x)}$. Suppose
$X=\und{X}(F)$, where $\und{X}$ is an $F$-variety, and $L$ is an algebraic line bundle on $\und{X}$. 
Then we can construct an $\C[t,t^{-1}]$-line bundle $|L|_t$ on $X$ as follows.
Let $(\und{U_i})$ be an open covering such that $L$ is glued from trivial line bundles on $\und{U_i}$ using transition functions
$f_{ij}\in \OO^*(\und{U_i})$. Then we glue $|L|_t$ out of sheaves of locally constant $\C[t,t^{-1}]$-valued functions on $U_i=\und{U_i}(F)$
using $|f_{ij}|_t$ as $\C[t,t^{-1}]^*$-valued transition functions.

For any open region $D\sub \C$, we have a natural embedding of $\C$-algebras, 
$$\C[t,t^{-1}]\hra\HH(D): t\mapsto q^{\kappa},$$
and the corresponding embedding 
$$\SS(X,|L|_t)\hra \SS(X,|L|_t\ot\HH(D))\simeq \SS^{hol}(X\times D,|L|^{\kappa}).$$

For $c\in \C^*$, we can consider the localization 
$$R_c:=\C[t,t^{-1}][\frac{1}{t-c}].$$
Then if $q^{\kappa}\neq c$ on the region $D\sub \C$, the above embeddings extend to
$$R_c\hra \HH(D), \ \ \SS(X,|L|_t\ot_{\C[t,t^{-1}]}R_c)\hra \SS^{hol}(X\times D,|L|^{\kappa}).$$

We have the following analog of Lemma \ref{sm-hom-char-lem}.

\begin{lemma}\label{sm-hom-char-R-lem}
Let $\chi:G\to \C[t,t^{-1}]^*$ be a locally constant homomorphism. Assume that there exists a closed embedding $F^*\hra G$, such that
$$\chi|_{F^*}(x)=t^{-v(x)}\cdot c^{v(x)},$$
for some $c\in \C^*$. Let $R$ be an $R_c$-algebra. Then viewing 
$\chi$ as a homomorphism $G\to R^*$, we have
$$H^{sm}_0(G,R_\chi)=0,$$
where $R_\chi$ is $R$ with the $G$-action given by $\chi$.
Assume in addition that
$G=F^*\ltimes F^n$, and $\chi$ is trivial on $F^n$.
Then $H^{sm}_*(G,R_\chi)=0$.
\end{lemma}

\begin{proof}
The proof is completely parallel to that of Lemma \ref{sm-hom-char-lem}.
We just have to use the fact that $t/c-1$ is invertible in $R_c$.
\end{proof}

We also have the following versions of Corollary \ref{G/H-cor} and Proposition \ref{hom-van-prop}.
The proofs are similar, so we omit them.

\begin{cor}\label{G/H-R-cor}
Let $R$ be a commutative $\C$-algebra.
For a closed subgroup $H\sub G$ a $G$-equivariant $R$-line bundle $L$ on $G/H$, one has
$$H^{sm}_i(G,\SS(G/H,L))\simeq H^{sm}_i(H,\chi_L\cdot (\De_H/\De_G|_H)),$$
where $\chi_L:H\to R^*$ is the character given by the action of $H$ on the fiber of $L$ at the point with the stabilizer $H$
(and we view $\De_H/\De_G|_H$ as taking values in $\C^*\sub R^*$).
\end{cor}

\begin{prop}\label{hom-van-R-prop}
Assume that the action of $G$ on $X$ is locally separated, and let $F$ be a $G$-equivariant sheaf of $R$-modules on $X$ (where $R$ is a commutative ring). 
Fix $i_0\ge 0$. Assume that for every $G$-orbit $\Om\sub X$, one has
$H_i(G,\SS(\Om,F|_{\Om}))=0$ for $i\le i_0$. Then $H_i(G,\SS(X,F))=0$ for $i\le i_0$.
\end{prop}

\subsection{An extension result for coinvariants of Schwartz sections}

We formulate a general result in the framework of $l$-spaces and $l$-groups, which extends the classical \cite[Thm.\ 6.9]{BZ} on invariant distributions.

\begin{theorem}\label{l-space-coinv-thm}
Let $X$ be an $l$-space equipped with action of an $l$-group $G$, $Z\sub X$ be a $G$-invariant closed subset.
Assume that the action of $G$ on $Z$ is constructible (i.e., the graph of the action in $Z\times Z$ is constructible).
Let $R$ be a commutative ring, $L$ a $G$-equivariant $R$-line bundle on $X$.
Assume that for every $z\in Z$, one has
\begin{equation}\label{hom-van-assumption-eq}
H_{i}^{sm}(G_z,\chi_{L,z}\cdot \De_{G_z}/\De_G|_{G_z})=0
\end{equation}
for $i=0,1$ (resp., for $i=0$), where $G_z$ is the stabilizer of $z\in Z$.
Then the natural map
\begin{equation}\label{X-Z-L-G-map}
\SS(X-Z,L)_{G}\to \SS(X,L)_{G}
\end{equation}
is an isomorphism (resp., surjective).
\end{theorem}

Note that by  \cite[Sec.\ 6.15]{BZ}, the condition of constructibility of action is satisfied if $G=\und{G}(F)$, $X=\und{X}(F)$, where $\und{X}$ is a variety over a local field $F$,
$\und{G}$ is a linear algebraic group over $F$ acting algebraically on $\und{X}$.

\begin{lemma}\label{constructible-lem} 
Assume that the action of $G$ on $Z$ is constructible. Then there exists a filtration $Z=Z^0\supset Z^1\supset\ldots\supset Z^k=\emptyset$
by closed $G$-invariant subsetss, such that the action of $G$ on $Z^i\setminus Z^{i+1}$ is locally separated (see Def.\ \ref{loc-sep-def}).
\end{lemma}

\begin{proof}
Set $\ov{Z}=Z/G$ and let us consider the diagonal $\De\sub \ov{Z}\times \ov{Z}$. As in \cite[Sec.\ 6.7]{BZ}, let us consider the finite filtration
$\De\supset \De^1\supset\ldots\supset \De^k=\emptyset$ by closed subsets, defined recursively as follows: $\De^{i+1}$ is the set of points $(\ov{z},\ov{z})\in \De^i$ such that
$\De^i$ is not closed in any neighborhood of $(\ov{z},\ov{z})$ (in $\ov{Z}\times \ov{Z}$).
By the identification $\De\simeq \ov{Z}$ we get a filtration $\ov{Z}\supset \ov{Z}^1\supset\ldots$ of $\ov{Z}$.
We claim that the corresponding filtration of $Z$ by $G$-invariant closed subspaces satisfies our requirement.
Indeed, by definition, $\De^i\setminus \De^{i+1}$ is locally closed in $\ov{Z}\times \ov{Z}$, hence, in $(Z^i\setminus Z^{i+1})\times (Z^i\setminus Z^{i+1})$. 
Therefore, the diagonal is locally closed in $(Z^i\setminus Z^{i+1})\times (Z^i\setminus Z^{i+1})$, and so the action of $G$ on $Z^i\setminus Z^{i+1}$ is
locally separated.
\end{proof}

\begin{proof}[Proof of Theorem \ref{l-space-coinv-thm}] 
The exact sequence of smooth $G$-representations
$$0\to \SS(X-Z,L)\to \SS(X,L)\to \SS(Z,L|_Z)\to 0$$
and the long exact sequence of smooth $G$-homology show that it is enough to prove the vanishing
$$H_i^{sm}(G,\SS(Z,L|_Z))=0$$
for $i=0,1$ (resp., $i=0$).

Using Lemma \ref{constructible-lem} we reduce to the case when the action of $G$ on $Z$ is locally separated.
Now applying Proposition \ref{hom-van-R-prop}, we are reduced to showing that for each $G$-orbit $\Om\sub Z$, one has
$$H_i^{sm}(G,\SS(\Om,L|_\Om))=0$$
for $i=0,1$ (resp., $i=0$).
It remains to apply Corollary \ref{G/H-R-cor} and the vanishing \eqref{hom-van-assumption-eq} for $i=0,1$ (resp., $i=0$).
\end{proof}


\subsection{Corollary for admissible stacks}

Let $\XX=[X/G]$ be an admissible stack of finite type over $F$ (where $F$ is a nonarchimedean local field, $G=\GL_N$),  $\LL$ a line bundle on $\XX$, $L$ the corresponding
$G$-equivariant line bundle on $X$.

We can associate with $L$ and with a number $\kappa\in\C$, a $G(F)$-equivariant $\C$-line bundle $|L|^{\kappa}$ on $X(F)$, and by passing to $G(F)$-coinvariants, the Schwartz
space $\SS(\XX,|\LL|^{\kappa})$ (see Sec.\ \ref{schw-stacks-sec}). On the other hand, we can define a $G(F)$-equivariant $\C[t,t^{-1}]$-line bundle $|L|_t$ on $X(F)$ and a $\C[t,t^{-1}]$-module
$\SS(\XX,|\LL|_t)$ defined as $G(F)$-coinvariants of $\SS(X(F),|L|_t)$ (see Sec.\ \ref{sh-coef-sec}). 

For an open substack $\UU\sub\XX$ and $\kappa\in\C$, there is a natural map
$$j_{\UU\to \XX,\LL,\kappa}:\SS(\UU,|\LL|^{\kappa})\to \SS(\XX,|\LL|^{\kappa}).$$
Also, for every commutative $\C[t,t^{-1}]$-algebra $R$, we can consider a natural map
$$j^R_{\UU\to \XX,\LL}:\SS(\UU,|\LL|_t\ot R)\to \SS(\XX,|\LL|_t\ot R).$$

For every point $x\in \XX$, the action of the group $\Aut(x)$ on the line bundle $\LL$ gives a character
$$\chi_{\LL,x}:\Aut(x)\to \G_m.$$
On the other hand, we have the algebraic modular character $\De^{alg}_{\Aut(x)}:\Aut(x)\to \G_m$.

For a character $\chi:\Aut(x)\to \G_m$ and a cocharacter $\la^\vee:\G_m\to \Aut(x)$, we define an integer $m=\lan\chi,\la^\vee\ran$ so that
$$\chi(\la^\vee(a))=a^m.$$

\begin{lemma}\label{surjective-lem-nonarch}
(i) Fix $\kappa\in\C$.
Assume that for every point $x\in \XX(F)\setminus \UU(F)$, 
there exists an algebraic homomorphism $\la^\vee:\G_m\to \Aut(x)$ (defined over $F$), such that
$$q^{\kappa\cdot \lan \chi_{\LL,x},\la^\vee\ran}\neq q^{-\lan\De^{alg}_{\Aut(x)},\la^\vee\ran},$$
where $q$ is the number of elements in the residue field. Then
$j_{\UU\to \XX,\LL,\kappa}$ is surjective. Assume in addition that for each $x\in \XX(F)\setminus \UU(F)$,
we have $\Aut(x)\simeq \G_m\ltimes \G_a^n$, where $\la^\vee:\G_m\to \Aut(x)$ given by the natural embedding. 
Then $j_{\UU\to \XX,\LL,\kappa}$ is an isomorphism.

\noindent
(ii) Let $R$ be a commutative $\C[t,t^{-1}]$-algebra. Assume that for every point  $x\in \XX(F)\setminus \UU(F)$, 
there exists an algebraic homomorphism $\la^\vee:\G_m\to \Aut(x)$, such that
$$t^{\lan \chi_{\LL,x},\la^\vee\ran}-q^{-\lan\De^{alg}_{\Aut(x)},\la^\vee\ran}\in R^*.$$
Then the map $j^R_{\UU\to \XX,\LL}$ is surjective. If in addition $\Aut(x)\simeq \G_m\ltimes \G_a^n$ and $\la^\vee:\G_m\to \Aut(x)$ is given by the natural embedding,
then  $j^R_{\UU\to \XX,\LL}$ is an isomorphism.

\noindent
(iii) Let $D\sub \C$ be an open region. Assume that for every point  $x\in \XX(F)\setminus \UU(F)$, 
there exists an algebraic homomorphism $\la^\vee:\G_m\to \Aut(x)$, such that
$$\operatorname{Re}(\kappa)\cdot \lan\chi_{\LL,x},\la^\vee\ran\neq \lan-\De^{alg}_{\Aut(x)},\la^\vee\ran$$
for every $\kappa\in D$. Then the map
$$j^{hol}_{\UU\to \XX,\LL}:\SS^{hol}(\UU\times D,|\LL|^{\kappa})\to\SS^{hol}(\XX\times D,|\LL|^{\kappa})$$
is surjective.
\end{lemma}

\begin{proof}
(i) Let $U\sub X$ be the $G$-invariant open subset corresponding to $\UU\sub\XX$. For $x\in X(F)$, we have $\Aut_{\XX}(x)=G_x\sub G$, the stabilizer of $x$.
Since $\De_{G(F)}=1$, by Theorem \ref{l-space-coinv-thm}, to check that $j_{\UU\to \XX,\LL,\kappa}$ surjective (resp., an isomorphism), we need to know that
$$H_i^{sm}(G_x(F),|\chi_{\LL,x}|^{\kappa}\cdot |\De^{alg}_{G_x}|)=0$$
for $x\not\in U(F)$ and $i=0$ (resp., $i\le 1$).
For $i=0$ we are just computing the coinvariants, and the assumption implies that the corresponding character $G_x(F)\to \C^*$ is nontrivial.
To compute higher homology we use the extra assumption on the structure of $\Aut(x)$ and apply Lemma \ref{sm-hom-char-lem}.

\noindent
(ii) This is parallel to (i) using Lemma \ref{sm-hom-char-R-lem}.

\noindent
(iii) This follows from (ii) with $R=\HH(D)$, viewed as a $\C[t,t^{-1}]$-algebra via $t\mapsto q^{\kappa}$.
\end{proof}

\section{Coinvariants for Schwartz sections: archimedean case}\label{coinv-arch-sec}

\subsection{Formulation of the extension result and the first reduction}

Now we assume that $F=\R$ or $\C$.
We want to prove the following version of \cite[Thm.\ B.0.2]{AG-sm-tr}\footnote{The modular characters appear in \cite[Thm.\ B.0.2]{AG-sm-tr} with the opposite sign
due to a different convention.}. 

\begin{theorem}\label{coinv-thm} 
Let $X$ be the set of $F$-points of a smooth algebraic $F$-variety equipped with an algebraic action of a connected linear algebraic group $G$ (over $F$) with
the Lie algebra $\fg$,
and let $Z\sub X$ be a $G$-invariant closed subvariety. Let $V$ (resp., $L_1,L_2$) be an algebraic $G$-equivariant vector bundle (resp., line bundles) over $X$.
As in Sec.\ \ref{Schwartz-var-sec}, consider $E_{\kappa}=V\ot |L_1|^{\kappa}\ot |L_2|$, where $\kappa\in \C$.
For $z\in Z$ we denote by $G_z\sub G$ the stabilizer subgroup, and by $\fg_z$ its Lie algebra.
Let $\chi_{L_i,z}:G_z\to \G_m$ denote the character of action on $L|_z$.
Let $D$ be either an open region in $\C$ or a single point $D=\{\kappa_0\}$.
Assume that for any $z\in Z$, $\kappa\in D$ and $m\ge 0$ we have
$$[|\chi_{L_1,z}|^{\kappa}\ot |\chi_{L_2,z}|\ot V|_z\ot S^m(N^\vee_{Gz})|_z\ot \De_{G_z}/\De_G|_{G_z}]_{\fg_{z,r}}=0$$
for some reductive subalgebra $\fg_{z,r}\sub \fg_z$, where $N^\vee_{Gz}$ is the conormal bundle to $Gz$
Then for any open $D'\sub D$ such that $\ov{D'}$ is compact and is contained in $D$, the map on coinvariants
$$[\SS^{hol}(X\times D,E_{\kappa})/\SS^{hol}((X-Z)\times D,E_{\kappa})]_{\fg}\to [\SS^{hol}(X\times D',E_{\kappa})/\SS^{hol}((X-Z)\times D',E_{\kappa})]_{\fg}$$
is zero. In the case $D=\{\kappa_0\}$, the conclusion is
$$[\SS(X,E_{\kappa_0})/\SS(X-Z,E_{\kappa})]_{\fg}=0$$
\end{theorem}

We follow the same steps as in \cite{AG-sm-tr}. 

First, we observe that if we have a filtration $Z_1\sub Z_2\sub\ldots Z_r=Z$ by $G$-invariant closed subvarieties then it is enough
to prove the result for each pair $(X-Z_{i-1},Z_i-Z_{i-1})$.

\begin{definition} An action of an algebraic group $G$ on an algebraic variety $Z$ is called {\it factorizable} if $Z$ is smooth, the schematic image $Y$ of the map 
$G\times X\to X\times X:(g,x)\mapsto (x,gx)$ is a smooth variety, and the map $G\times Z\to Y$ is smooth.
\end{definition}

There exists a filtration $Z_1\sub\ldots\sub Z_r=Z$ be $G$-invariant closed subvarieties, such that the action of $G$ on $Z_i-Z_{i-1}$ is factorizable
(see \cite[Thm.\ B.0.11]{AG-sm-tr}). 
Thus, we can assume that $Z$ is smooth and the action of $G$ on $Z$ is {\it factorizable}.

By an analog of Borel's theorem (see Prop.\ \ref{Borel-thm-prop}) together with Lemma \ref{filtr-lem},
Theorem \ref{coinv-thm} now reduces to the following result (which is an analog of \cite[Thm.\ B.0.12]{AG-sm-tr}), which has to be applied to
$X=Z$ and the bundle $S^m(N^\vee_Z)\ot (V\ot |L_1|^{\kappa}\ot |L_2|)|_Z$.
Note that to deduce Theorem \ref{coinv-thm} from Theorem \ref{coinv-fact-thm} we use reductivity of
$\fg_{z,r}$, as we have to pass from the space $N^\vee_{Gz}|_z$ to its subspace $N^{\vee}_Z|_z$.\footnote{There seems to be a gap in the proof of \cite[Thm.\ B.0.2]{AG-sm-tr} at a similar step.}

\begin{theorem}\label{coinv-fact-thm}
Let $(X,G,V,L_1,L_2,D)$ be as in Theorem 
and assume in addition that the action of $G$ on $Z$ is factorizable.
Assume that for any $z\in Z$ and $\kappa\in D$ we have
$$[|\chi_{L_1,z}|^{\kappa}\ot |\chi_{L_2,z}|\ot V|_z\ot \De_{G_z}/\De_G|_{G_z}]_{\fg_{z}}=0,$$
where $\fg_z$ is the Lie algebra of $G_z$.
Then 
$$\SS^{hol}(X\times D,E_{\kappa})_{\fg}=0.$$
In the case $D=\{\kappa_0\}$, the conclusion is
$$\SS^{hol}(X,E_{\kappa_0})_{\fg}=0.$$
\end{theorem}

\subsection{Description of coinvariants}

\begin{lemma}\label{Frechet-image-lem} 
Let $\fg$ be a finite dimensional Lie algebra (over $\R$ or $\C$), and
let $V$ be a nuclear Frechet space with a continuous $\fg$-action, such that $\fg\cdot V$ is closed in $V$.
Then for a nuclear Frechet space $W$ (with the trivial $\fg$-action), we have
$$\fg\cdot (V\hat{\ot} W)=(\fg\cdot V)\hat{\ot} W.$$
\end{lemma}

\begin{proof}
By definition $\fg\cdot (V\hat{\ot} W)$ is the image of the action map
$$\fg\ot (V\hat{\ot} W)\to V\hat{\ot} W,$$
which can be identified with the map $(\fg\ot V)\hat{\ot} W\to V\hat{\ot} W$
induced by the action map $\fg\ot V\to V$. By assumption, the latter map factors as the composition
of the surjection $\fg\ot V\to V_0:=\fg\cdot V$ and a closed embedding $V_0\to V$.
Now the assertion follows from the fact that the induced maps
$(\fg\ot V)\hat{\ot} W\to V_0\hat{\ot} W$ and $V_0\hat{\ot} W\to V\hat{\ot} W$ are
still surjective and injective, respectively.
\end{proof}

Let $a:G\times X\to X$ be the action map $(g,x)\mapsto gx$, $p:G\times X\to X$ the projection.
The $G$-equivariant structure on $E_{\kappa}$ gives a (tempered) map of bundles on $G\times X$,
$p^*(E_{\kappa})\rTo{\sim} a^*(E_{\kappa})$, and hence a map
$$\a:p^!(E_{\kappa})\rTo{\sim} a^!(E_{\kappa}),$$ 
where for a smooth morphism $f$ we set $f^!(F):=f^*(F)\ot |\om_f|$
(here we use the identifications $\om_p\simeq \om_a\simeq p_G^*\om_G$).

Similar to \cite{AG-sm-tr}, we consider the kernel of the integration map,
$$\SS^{hol}(G\times X\times D,p^!E_{\kappa})_{0,X}:=\ker(p_*:\SS^{hol}(G\times X\times D,p^!E_{\kappa})\to \SS^{hol}(X\times D,E_{\kappa})).$$
and then we consider the image of this subspace under the above map $\a$,
$$\SS^{hol}(G\times X\times D,a^!E_{\kappa})_{0,X,a}:=\a(\SS^{hol}(G\times X\times D,p^!E_{\kappa})_{0,X}).$$


\begin{prop}\label{coinv-descr-prop} 
We have
$$\fg\cdot \SS^{hol}(X\times D,E_{\kappa})=a_*(\SS^{hol}(G\times X\times D,a^*E_{\kappa}\ot|\om_a|)_{0,X,a}).$$
\end{prop}

\begin{proof}
The integration map 
$$a_*:\SS^{hol}(G\times X\times D,a^*E_{\kappa}\ot |\om_a|)\to \SS^{hol}(X\times D,E_{\kappa})$$
is surjective (see Lemma \ref{surj-int-lem}) and is compatible with $G$-actions, where $G$ acts on the $G$-coordinate in $G\times X$.
It follows that $\fg\cdot \SS^{hol}(X\times D,E_{\kappa})$ is the image of
$\fg\cdot \SS^{hol}(G\times X\times D,a^*E_{\kappa}\ot|\om_a|)$.
Furthermore, the isomorphism
$$p^*E_{\kappa}\ot |\om_p|\simeq p^*E_{\kappa}\ot |p_G^*\om_G|\simeq a^*E_{\kappa}\ot |\om_a|$$
is compatible with $G$-actions, where $p^*E_{\kappa}$ has a trivial $G$-action.
Now we have an isomorphism of $\fg$-representations,
$$\SS^{hol}(G\times X\times D,p^*E_{\kappa}\ot |p_G^*\om_G|)\simeq \SS(G,|\om_G|)\hat{\ot} \SS^{hol}(X,E_{\kappa}).$$
Recall that $\fg\cdot \SS(G,|\om_G|)$ coincides with the (closed) subspace $\SS(G,|\om_G|)_0$ of $f$ with $\int_G f=0$
(see \cite{AG-sm-tr}).
This gives an identification
$$\fg\cdot \SS^{hol}(G\times X\times D,p^*E_{\kappa}\ot |p_G^*\om_G|)\simeq \SS(G,|\om_G|)_0\hat{\ot} \SS^{hol}(X,E_{\kappa})$$
(see Lemma \ref{Frechet-image-lem}). 
It remains to observe that the completed tensored product of the exact sequence
$$0\to \SS(G,|\om_G|)_0\to \SS(G,|\om_G|)\rTo{\int_G} \C\to 0$$
with $\SS^{hol}(X,E_{\kappa})$ gives the exact sequence
$$0\to \SS^{hol}(G\times X\times D,p^*E_{\kappa}\ot |p_G^*\om_G|)_{0,X}\to \SS^{hol}(G\times X\times D,p^*E_{\kappa}\ot |p_G^*\om_G|)\rTo{p_*}
\SS^{hol}(X\times D,E_{\kappa})\to 0.$$
\end{proof}


Next, we will consider the case of a smooth group scheme over $X$.

\begin{prop}\label{H-equiv-prop}
Let $\pi:H\to X$ be a smooth group scheme over $X$, $V_{\kappa}$ a family of $H(F)$-equivariant vector bundles over $X(F)$ of the form $P\ot |L_1|^{\kappa}\ot |L_2|$, where
$\kappa\in D_0$.
Assume that for every $x\in X$ and $\kappa\in D\sub D_0$, one has $(V_{\kappa}|_x)_{\fh|_x}=0$. Then the composition
$$\SS^{hol}(H\times D,\pi^*V_{\kappa}\ot |\om_\pi|)_{0,X}\rTo{\a}\SS^{hol}(H\times D,\pi^*V_{\kappa}\ot |\om_\pi|)\rTo{\pi_*}\SS^{hol}(X\times D,V_{\kappa}),$$
induced by the $H$-action map $\a:\pi^*V_{\kappa}\to \pi^*V_{\kappa}$ followed by the integration along $\pi:H\to X$,
is surjective.
\end{prop}

\begin{proof}
{\bf Step 1}. The surjective composed map 
$$\SS^{hol}(H\times D,\pi^*V_{\kappa}\ot |\om_\pi|)\rTo{\a}\SS^{hol}(H\times D,\pi^*V_{\kappa}\ot |\om_\pi|)\rTo{\pi_*}\SS^{hol}(X\times D,V_{\kappa}),$$
is a map of $\SS(X,\fh)$-representations, where the action on the source considers $V_{\kappa}$ with the trivial $H$-action (cf. the proof of \cite[Cor.\ 9.17]{AG-sm-tr}).

\noindent
{\bf Step 2}. The image of the map $\SS(X,\fh)\ot \SS^{hol}(H\times D,\pi^*V_{\kappa}\ot |\om_\pi|)\to \SS^{hol}(H\times D,\pi^*V_{\kappa}\ot |\om_\pi|)$
associated with the trivial $H$-action on $V_{\kappa}$ is contained in $\SS^{hol}(H\times D,\pi^*V_{\kappa}\ot |\om_\pi|)_{0,X}$.

Indeed, by partition of unity, we can assume that $V_{\kappa}$ is trivial. We have the following identifications:
$$\SS^{hol}(H\times D,|\om_\pi|)\simeq \SS(H,|\om_\pi|)\hat{\ot} \SS^{hol}(D), \ \ \SS^{hol}(X\times D)\simeq \SS(X)\hat{\ot} \SS^{hol}(D),$$
$$\SS^{hol}(H\times D,|\om_\pi|)_{0,X}\simeq \SS(H,|\om_\pi|)_{0,X}\hat{\ot} \SS^{hol}(D),$$
where the latter identification is obtained by looking at the completed tensor product with $\SS^{hol}(D)$ of the sequence
$$0\to \SS(H,|\om_\pi|)_{0,X}\to \SS(H,|\om_\pi|)\rTo{\pi_*}\SS(X)\to 0.$$
By \cite[Lem.\ B.1.14]{AG-sm-tr}, we have an inclusion
$$\im(\SS(X,\fh)\ot \SS(H,|\om_\pi|)\to \SS(H,|\om_\pi|))\sub \SS(H,|\om_\pi|)_{0,X}.$$
This implies that the composition
$$\SS(X,\fh)\hat{\ot} \SS(H,|\om_\pi|)\to \SS(H,|\om_\pi|)\rTo{\pi_*}\SS(X)$$
is zero. Tensoring with $\SS^{hol}(D)$, we deduce that the composition
$$\SS(X,\fh)\hat{\ot} \SS(H,|\om_\pi|)\hat{\ot} \SS^{hol}(D)\to \SS(H,|\om_\pi|)\hat{\ot} \SS^{hol}(D) \rTo{\pi_*\ot \id}\SS(X)\hat{\ot} \SS^{hol}(D),$$
is zero, which gives the assertion in view of the above identifications.

\noindent
{\bf Step 3}. The map $\SS(X,\fh)\ot \SS^{hol}(X\times D,V_{\kappa})\to \SS^{hol}(X\times D,V_{\kappa})$ (induced by the $H$-action on $V_{\kappa}$) is surjective.
The action of $\fh$ on $V_{\kappa}$ gives a morphism of vector bundles
$$\a_{\fh}:\fh\ot V_{\kappa}\to V_{\kappa}$$
over $X\times D$, and our assumption implies that it is surjective. By the partition of unity, we can assume that $X$ is affine, and can trivialize $P$, $L_1$ and $L_2$.
and $\fh$ as a vector bundle. Let us write $\fh=\fh_0\ot \OO$, $V_{\kappa}=V_0\ot \OO$, where $\fh_0$ and $V_0$ are finite-dimensional vector spaces.
Then $\a_{\fh}$ can be viewed as a surjective algebraic morphism 
$$\fh_0\ot V_0\ot \OO_{X\times D}\to V_0\ot \OO_{X\times D}$$ 
over $X\times D$ (the dependence on $\kappa$ will be linear). 

We have identifications
$$\SS(X,\fh)\simeq \fh_0\ot \SS(X), \ \ \SS^{hol}(X\times D,V_{\kappa})\simeq V_0\ot \SS^{hol}(X\times D),$$
so the map we are interested in is obtained as the composition
$$\fh_0\ot \SS(X)\ot V_0\ot \SS^{hol}(X\times D)\to \fh_0\ot V_0\ot \SS^{hol}(X\times D)\rTo{\a_{\fh}} V_0\ot \SS^{hol}(X\times D),$$
where the first arrow is induced by the natural surjective map $\SS(X)\ot \SS^{hol}(X\times D)\to \SS^{hol}(X\times D)$.
Thus, it suffices to prove surjectivity of the map
\begin{equation}\label{th0-V0-surj-map}
\fh_0\ot V_0\ot \SS^{hol}(X\times D)\rTo{\a_{\fh}} V_0\ot \SS^{hol}(X\times D).
\end{equation}

To this end we can choose a section 
$$s:V_0\ot \OO_{X\times D}\to \fh_0\ot V_0\ot \OO_{X\times D}$$
of $\a_{fh}$. Then $s$ gives rise to a map
$$V_0\ot \SS^{hol}(X\times D)\to \fh_0\ot V_0\ot \SS^{hol}(X\times D),$$
which is a section of \eqref{th0-V0-surj-map}, proving its surjectivity.

\noindent
{\bf Step 4}. Now the assertion follows formally from Steps 1, 2 and 3.
\end{proof}

\subsection{Representation of torsors}

We use the terminology of \cite[Sec.\ B.1.4]{AG-sm-tr} concerning (families of) torsors.

We consider a smooth morphism of $F$-varieties, $\pi:T\to Y$ with the torsor structure over $Y$. This means that we are given a map $m:T\times_Y T\times_Y T\to T$
which is locally of the form $m(x,y,z)=z(z^{-1}x)(z^{-1}y)$ for some group structure. For every $t\in T$ we denote by $H_t$ the fiber $\pi^{-1}(\pi(t))$ with the unique group structure
such that $t$ is the neutral element. In the archimedean case we denote by $\fh_t$ the Lie algebra of $H_t$.
An (algebraic) representation of $T$ is a pair of (algebraic) vector bundles $V_1$, $V_2$ on $Y$, together with a morphism 
$\a:\pi^*V_1\to \pi^*V_2$ such that for every $t_0\in T$, the map
$t\mapsto \a_t\a_{t_0}^{-1}:V_2|_{\pi(t_0)}\to V_2|_{\pi(t_0)}$ gives an action of $H_{t_0}$ on $V_2|_{\pi(t_0)}$. 
There is a natural operation of tensor product of representations of $T$.

We have the following Corollary from Proposition \ref{H-equiv-prop} (an analog of \cite[Cor.\ B.1.27]{AG-sm-tr}).

\begin{cor}\label{torsor-cor}
Let $\pi:T\to Y$ be an algebraic torsor over $F$, $(V_{1,\kappa},V_{2,\kappa},\a)$ a family of representations of $T(F)$, obtained as the tensor product of
an algebraic representation $(P_1,P_2)$, with $1$-dimensional representations of the form $(|L_1|^{\kappa},|L_2|^{\kappa})$ and 
$(|M_1|,|M_2|)$ (where $(L_1,L_2)$ and $(M_1,M_2)$ are algebraic representations of $T$). 
Assume that for every $t\in T$ and $\kappa\in D$, one has $(V_{2,\kappa}|_{p(t)})_{\fh_t}=0$. Then the composition
\begin{equation}\label{main-surj-eq}
\SS^{hol}(T\times D,\pi^*V_{1,\kappa}\ot |\om_\pi|)_{0,Y}\rTo{\a}\SS^{hol}(T\times D,\pi^*V_{2,\kappa}\ot |\om_\pi|)\rTo{\pi_*}\SS^{hol}(Y\times D,V_{\kappa}),
\end{equation}
is surjective.
\end{cor}

\begin{proof} If there is a section $\si:Y\to T$ then $T$ becomes a family of smooth groups and the assertion
follows from Proposition \ref{H-equiv-prop}. The general case follows from this by choosing local sections and a partition of unity.
\end{proof}

\begin{proof}[Proof of Theorem \ref{coinv-fact-thm}]
Let $Y\sub X\times X$ denote the image of the map $(p,a):G\times X\to X\times X: (g,x)\mapsto (x,gx)$. 
In other words, $Y$ is the variety of pairs $(x_1,x_2)$ such that $x_1$ and $x_2$ are in the same $G$-orbit.
By assumption, $Y$ is smooth and the map 
$$b:G\times X\to Y:(g,x)\mapsto (x,gx)$$ 
is smooth. Hence, the projections $p_1,p_2:Y\to X$ are also smooth.

Via the map $b$, we can view $G\times X$ as a torsor over $Y$. Note that $b^{-1}(x_1,x_2)$ is identified with $\{g\in G \ |\ gx_1=x_2\}$.
For each $(g,x)\in G\times X$, the corresponding group $H_{g,x}$ is canonically identified with $H_x\simeq H_{gx}$.

The pair of bundles on $Y$,
$$V_{i,\kappa}=p_i^*E_{\kappa}\ot |\om_{p_i}|, \ i=1,2,$$
has a natural structure of a representation over the torsor $G\times X$. Namely, 
for $(x_1,x_2)\in Y$, and an element $g\in G$ such that $gx_1=x_2$, we use the map
$$V_{1,\kappa}|_{(x_1,x_2)}\simeq E_{\kappa}|_{x_1}\ot |\om_{\Om}||_{x_2}\rTo{\a_g\ot \a_{g^{-1}}} E_{\kappa}|_{x_2}\ot |\om_{\Om}||_{x_1},$$
induced by the action of $g$ on $E_{\kappa}$ and by the action of $g^{-1}$ on $\om_{\Om}$, where $\Om$ is the $G$-orbit containing $x_1$ and $x_2$.
The corresponding representation of $H_x$ on $V_{2,\kappa}|_x$ can be identified with $E_{\kappa}\ot \De_{H_x}\ot \De_G^{-1}|_{H_x}$. 
Thus our assumption implies that for every $x\in X$, one has $(V_{2,\kappa}|_x)_{\fh_x}=0$.

Applying Corollary \ref{torsor-cor} we deduce surjectivity of the composition 
\eqref{main-surj-eq}.
On the other hand, by Proposition \ref{coinv-descr-prop}, we need to prove surjectivity of the composition
$$\SS^{hol}(G\times X\times D,p^!E_{\kappa})_{0,X}\hra \SS^{hol}(G\times X\times D,p^!E_{\kappa})\rTo{\a} \SS^{hol}(G\times X\times D,a^!E_{\kappa})\rTo{a_*}\SS^{hol}(X\times D,E_{\kappa}).$$
Since $a=p_2\circ b$, and the map $p_2:Y\to X$ is smooth and surjective on $F$-points, by Lemma \ref{surj-int-lem}, it is enough to prove surjectivity of the similar composition
with $a_*$ replaced by the map
$$b_*:\SS^{hol}(G\times X\times D,a^!E_{\kappa})\to \SS^{hol}(Y\times D,p_2^!E_\kappa).$$
Finally, we observe that we have an inclusion
$$\SS^{hol}(G\times X\times D,b^*V_{1,\kappa}\ot |\om_b|)_{0,Y}\sub \SS^{hol}(G\times X\times D,p^!E_{\kappa})_{0,X}$$
(the first space is the kernel of $b_*$, while the second is the kernel of $p_*$), and 
a commutative diagram
\begin{diagram}
\SS^{hol}(G\times X\times D,p^!E_{\kappa})&\rTo{\a}&\SS^{hol}(G\times X\times D,a^!E_{\kappa})\\
\dTo{\sim}&&\dTo{\sim}\\
\SS^{hol}(G\times X\times D,b^!V_{1,\kappa})&\rTo{\a}&\SS^{hol}(G\times X\times D,b^!V_{2,\kappa}).
\end{diagram}
Hence, the assertion follows from surjectivity of \eqref{main-surj-eq}.
\end{proof}

\subsection{Corollary for admissible stacks}



Let $\XX$ be an admissible stack of finite type over $F$,  $\LL$ a line bundle on $\XX$.

Let $\UU\sub \XX$ an open substack. For every $\kappa\in\C$, there is a natural map 
$$j_{\UU\to \XX,\LL,\kappa}:\SS(\UU,|\LL|^{\kappa})\to \SS(\XX,|\LL|^{\kappa})$$
and for an open region $D\sub\C$, a map 
$$j^{hol}_{\UU\to \XX,\LL,D}:\SS^{hol}(\UU\times D,|\LL|^{\kappa})\to \SS^{hol}(\XX\times D,|\LL|^{\kappa})$$

As before, for every point $x\in \XX$, we consider
the character
$$\chi_{\LL,x}:\Aut(x)\to \G_m.$$
and the algebraic modular character $\De^{alg}_{\Aut(x)}$ of $\Aut(x)(F)$.

\begin{lemma}\label{surjective-lem}
Let $D$ be either an open region in $\C$ or $D=\{\kappa_0\}$.
Assume that for every point $x\in \XX\setminus \UU$ and $\kappa\in D$, 
there exists an algebraic homomorphism $\la^\vee:\G_m\to \Aut(x)$, such that for any character $a\mapsto a^m$ of $\G_m$ appearing in the action on $S^\bullet H^0T^*_{\XX}|_x$ one has
$$\kappa\cdot \lan \chi_{\LL,x},\la^\vee\ran+m\neq -\lan\De^{alg}_{\Aut(x)},\la^\vee\ran.$$
Then 
for any open $D'\sub D$ such that $\ov{D'}$ is compact and is contained in $D$,
the image of the restriction map
$$\SS^{hol}(\XX\times D,|\LL|^{\kappa})\to \SS^{hol}(\XX\times D',|\LL|^{\kappa})$$
is contained in the image of $j^{hol}_{\UU\to \XX,\LL,D'}$.
In the case $D=\{\kappa_0\}$, the conclusion is that $j_{\UU\to\XX,\LL,\kappa_0}$ is surjective.
\end{lemma}

\begin{proof}
We can assume that $\XX=[X/G]$, where $G=\GL_N$, and $\UU=[U/G]$, where $U\sub X$ is open.
Let $L$ denote the pull-back of $\LL$ to $X$.
We have $\Aut_{\XX}(x)=G_x\sub G$, the stabilizer subgroup of $x\in X$. 
Theorem \ref{coinv-thm} implies that the top horizontal arrow in the diagram 
\begin{diagram}
\SS^{hol}(X\times D,|L|^{\kappa})_{\fg}/\im \SS^{hol}(U\times D,|L|^{\kappa})_{\fg}&\rTo{}&
\SS^{hol}(X\times D',|L|^{\kappa})_{\fg}/\im\SS^{hol}(U\times D',|L|^{\kappa})_{\fg}\\
\dTo{}&&\dTo{}\\
\SS^{hol}(\XX\times D,|\LL|^{\kappa})/\im \SS^{hol}(\UU\times D,|\LL|^{\kappa})&\rTo{}&
\SS^{hol}(\XX\times D',|\LL|^{\kappa})/\im \SS^{hol}(\UU\times D',|\LL|^{\kappa})
\end{diagram}
is zero.
Now we observe that the vertical arrows in this diagram are surjective by Lemma \ref{Lie-group-coinv-lem}. Hence, the bottom horizontal arrow is also zero.
\end{proof} 



%

\section{Integration over orbits}\label{orbit-int-sec}

\subsection{Setup and two extension results}\label{integration-ext-sec}

In this section $F$ is either archimedean or non-archimedean.

Let $X$ be a smooth $G$-scheme over $F$  (where $G=\GL_N$), $L$ a $G$-equivariant line bundle
on $X$ representing a line bundle $\LL$ on the stack $[X/G]$, and let $U_0\sub X$ be an open $G$-invariant subscheme such that $G$ acts freely on $U_0$ and
the quotient $U_0/G$ exists. Let us choose a trivialization 
$\vol_{\fg}\in |{\bigwedge}^{top}(\fg)^{-1}|$.

First, we observe that for each $\kappa\in\C$,
there is a well defined map
$$\pi_{U_0,\kappa}:\SS(U_0(F),|L|^{\kappa})_{G(F)}\to \SS((U_0/G)(F),|L_0|^{\kappa}),$$
where $L_0$ is the descent of $L$ to $U_0/G$,
given by absolutely convergent integrals
\begin{equation}\label{orbit-integral-eq}
\pi_{U_0,\kappa}(\varphi)(x)=\int_{\Om_x(F)} \varphi\cdot \vol_{\fg},
\end{equation}
where $\Om_x\sub X$ is the $G$-orbit of $x\in U_0$ (which is closed in $U_0$).

We are interested in finding conditions under which the above map extends to a map
\begin{equation}\label{lc-kappa-map}
\pi_{X,U_0,\kappa}:\SS(X(F),|L|^{\kappa})_{G(F)}\to C^\infty((U_0/G)(F),|L_0|^{\kappa}),
\end{equation}
given by the same formula \eqref{orbit-integral-eq}
(where in the nonarchimedean case $C^\infty$ refers to locally constant functions).
For technical reasons we will define $\pi_{U_0,\kappa}(\varphi)$ as distributions on $(U_0/G)(F)$, so actually
we will be able to control the integrals \eqref{orbit-integral-eq} only away from a set of orbits of measure zero.

We will consider the space of distributions $\DD((U_0/G)(F),|L_0|^{\kappa})$ acting on sections of $|L_0|^{-\kappa}\ot |\om|$,
so that we have an inclusion 
$$C^\infty((U_0/G)(F),|L_0|^{\kappa})\sub \DD((U_0/G)(F),|L_0|^{\kappa}).$$

\begin{definition}
We say that the pair $(U_0,X)$ is {\it $|L|^{\kappa}$-nice} if for every $\varphi\in \SS(X(F),|L|^{\kappa})$, and every 
$\psi\in \SS((U_0/G)(F),|L_0|^{-\kappa}\ot |\om|)$, the integral 
\begin{equation}\label{distr-integral-eq}
\pi_{X,U_0,\kappa}(\varphi)(\psi):=\int_{U_0(F)} \psi\cdot \varphi|_{U_0(F)}\cdot \vol_{\fg}
\end{equation} 
is absolutely convergent, and the corresponding distribution
$\pi_{X,U_0,\kappa}(\varphi)\in \DD((U_0/G)(F),|L_0|^{\kappa})$ belongs to the subspace $C^\infty((U_0/G)(F),|L_0|^{\kappa})$.
\end{definition}

It is easy to see using Fubini's theorem that in the situation of the above definition the integrals \eqref{orbit-integral-eq} are absolutely convergent away from
a set of orbits of measure zero. It is plausible that if $(U_0,X)$ is $|L|^{\kappa}$-nice then in fact the values of the obtained $C^\infty$-sections $\pi_{X,U_0,\kappa}(\varphi)$ are given
by absolutely convergent integrals \eqref{orbit-integral-eq} but we do not know how to prove this (and this is not essential for us).

From now on we specialize to the case $\LL=\om_{\XX}$, so $L=\om_X\ot {\bigwedge}^{top}(\fg)$. Note that since the action of $G$ on ${\bigwedge}^{top}(\fg)$ is trivial,
we can (and will) identify $L$ with $\om_X$ as a $G$-equivariant bundle.

Assume we have some  open $G$-invariant subschemes $U_0\sub U\sub X$ such that $G$ acts freely on $U_0$ and the quotient $U_0/G$ exists. 


As before, for $x\in X(F)$, $G_x=\Aut_{[X/G]}(x)$ is the stabilizer of $x$, and $\chi_{\om,x}:G_x\to\G_m$ is the character given by the action on $\om_X|_x$.

\begin{lemma}\label{extension-lem1}
Fix some real numbers $a<1$, $b>1$.
Assume that for any real number $k$ in the interval $(a,b)\sub \R$, the pair $(U_0,U)$ is $|\om|^k$-nice.
Assume in addition that 
\begin{itemize}
\item either $F$ is non-archimedean, and
for any $x\in (X\setminus U)(F)$ there exists a cocharacter $\la^\vee:\G_m\to G_x$ such that
$$k\cdot\lan \chi_{\om,x},\la^\vee\ran\neq -\lan \De^{alg}_{G_x},\la^\vee\ran$$
for any $k\in (a,b)$;
\item
or $F$ is archimedean, and
for any $x\in (X\setminus U)(F)$ there exists a cocharacter $\la^\vee:\G_m\to G_x$, such that for any character $a\mapsto a^m$ of $\G_m$ appearing in $S^\bullet H^0T^*_{[X/G]}|_x$,
$$k\cdot\lan \chi_{\om,x},\la^\vee\ran+m \neq -\lan \De^{alg}_{G_x},\la^\vee\ran$$
for any $k\in (a,b)$.
\end{itemize}
Then the pair $(U_0,X)$ is $|\om|^{\kappa}$-nice for any $\kappa$ in the strip $a<\operatorname{Re}(\kappa)<b$.
\end{lemma}

\begin{proof}
By shrinking $U_0/G$ we can assume that there exists a nowhere vanishing
$G$-invariant top form $\eta_0$ on $U_0/G$. 
On the other hand, we can cover $X$ with open subsets $U_i$ such that there exist nowhere vanishing top forms $\eta_i$ on $U_i$, so that $\eta_i\cdot \vol_{\fg}^{-1}$
are trivializations of $\om_X|_{U_i}\ot {\bigwedge}^{top}(\fg)$.

It is enough to consider sections of $\SS(X(F),|L|^{\kappa})$ supported on some $U_i$. Using the trivializations of $L|_{U_i}$, we can write these sections as
$$\varphi_\kappa:=\varphi\cdot |\eta_i\cdot \vol_{\fg}^{-1}|^{\kappa},$$
where $\varphi\in \SS(U_i(F))$, so
the corresponding integrals can be written as
\begin{equation}\label{Ui-integral-eq}
\pi_{X,U_0,\kappa}(\varphi_\kappa)(\psi |\eta_0|^{1-\kappa})=
\int_{U_0(F)\cap U_i(F)} \psi\cdot \varphi\cdot |\frac{\eta_i}{\vol_{\fg}\eta_0}|^{\kappa}\cdot |\eta_0|\cdot \vol_{\fg},
\end{equation}
where $\psi\in \SS((U_0/G)(F))$. Note that here $\eta_i/(\vol_{\fg}\eta_0)$ is a non-vanishing function on $(U_i\cap U_0)(F)$.

Furthermore, it is enough to consider the Schwartz sections as above with $\varphi$
non-negatve (since any $\varphi$ can be expressed as a difference of non-negative Schwartz functions).
From now one, we fix such $\varphi\in\SS(U_i(F))$.

\medskip

\noindent
{\bf Step 1}. We observe that for $\operatorname{Re}(\kappa)=1$, 
the integral \eqref{Ui-integral-eq} absolutely converges, since the measure $\varphi\cdot |\eta_i|$ has finite volume. 
Hence, \eqref{Ui-integral-eq} defines a well defined distribution $\pi_{X,U_0,\kappa}(\varphi_{\kappa})\in
 \DD((U_0/G)(F),|\om|^{\kappa})$.


\medskip

\noindent
{\bf Step 2}. It is enough to prove the relevant absolute convergence for $\kappa$ in any slightly smaller strip $a'<\operatorname{Re}(\kappa)<b'$,
where $a'=a+\eps$, $b'=b-\eps$. Let us pick a bounded open domain $D'$ contained in this strip such that $D'\cap \R=(a',b')$.
By Lemma \ref{surjective-lem} in the archimedean case or by Lemma \ref{surjective-lem-nonarch}(iii) in the non-archimedean case,
there exists an element
$$\wt{\varphi}_{\kappa}\in \SS(U(F)\times D',|\om|^{\kappa})_{G(F)}$$ 
lifting $\varphi_\kappa$.
Note that
$$\pi_{U,U_0,\kappa}(\wt{\varphi}_{\kappa})\in C^\infty((U_0/G)(F),|\om|^{\kappa})\sub  \DD((U_0/G)(F),|\om|^{\kappa})$$ 
depends holomorphically on $\kappa\in D'$.


\medskip

\noindent
{\bf Step 3}. We know that for $\operatorname{Re}(\kappa)=1$,
the integrals defining the distribution $\pi_{X,U_0,\kappa}(\varphi_{\kappa})$
are absolutely convergent 
and we have an equality of twisted distributions on $(U_0/G)(F)$.
$$\pi_{U,U_0,\kappa}(\wt{\varphi}_{\kappa})=\pi_{U,U_0,\kappa}(\varphi_{\kappa})=\pi_{X,U_0,\kappa}(\varphi_\kappa).$$
In particular, for $\operatorname{Re}(\kappa)=1$, the image of $\pi_{X,U_0,\kappa}$ is contained in $C^\infty((U_0/G)(F),|\om|^{\kappa})$.

\medskip

\noindent
{\bf Step 4}. Now we arrive at the following situation: we have a holomorphic map 
$$\{\kappa\in\C \ |\ a'<\operatorname{Re}(\kappa)<b'\}\to C^\infty((U_0/G)(F))\sub \DD((U_0/G)(F)):\kappa\mapsto \pi_{U,U_0,\kappa}(\wt{\varphi}_{\kappa}),$$
such that for $\operatorname{Re}(\kappa)=1$ it is given by absolutely converging integrals $\pi_{X,U_0,\kappa}(\varphi_{\kappa})(\psi |\eta_0|^{1-\kappa})$.

Now by Lemma \ref{anal-cont-lem} below, we conclude that for any non-negative $\psi$ the integral
\eqref{Ui-integral-eq} converges absolutely
for all $\kappa$ in the strip $a'<\operatorname{Re}(\kappa)<b'$, and agrees with $\pi_{U,U_0,\kappa}(\wt{\varphi}_{\kappa})(\psi |\eta_0|^{1-\kappa})$.
In particular, the distribution it defines belongs to $C^\infty((U_0/G)(F))$.
\end{proof}

The proof of the following Lemma is due to Pavel Etingof.

\begin{lemma}\label{anal-cont-lem}
Let $Y$ be a smooth variety over a local field $F$, $f$ an invertible function on $Y$, $\mu$ a positive measure on $Y(F)$.
Consider the integrals depending on $s\in C$,
$$I(s)=\int_{Y(F)}|f|^s \mu.$$
Assume that $I(s_0)<\infty$ for some $s_0\in\R$, so that $I(s)$ is well defined for $\operatorname{Re}(s)=s_0$, and  
there exists an analytic continuation of $I(s)$ into an open neighborhood $U\sub \C$ of $s_0$.
Suppose $s_0\sub (a,b)\sub U\cap \R$.
Then $I(s)<\infty$ for $a<\operatorname{Re}(s)<b$.
\end{lemma}

\begin{proof} Without loss of generality we can assume that $s_0=0$. Let us write $I(s)=I_+(s)+I_-(s)$, where
$$I_+(s)=\int_{|f|\ge 1}|f|^s\cdot \mu, \ \ I_-(s)=\int_{|f|<1}|f|^s \mu.$$
Since $I(0)<\infty$, the integral $I_+(s)$ (resp., $I_-(s)$) converges absolutely for $\Re(s)\le 0$ (resp., $\Re(s)\ge 0$).

It is enough to prove that $I_+(b)<\infty$ and $I_-(a)<\infty$. Let us prove the convergence of $I_-(a)$ (the proof for $I_+(b)$ is analogous).
The existence of an analytic continuation of $I(s)$ into $a<\Re(s)<\eps$ (where $\eps>0$) implies that $I_-(s)$ extends to a function on $a<\Re(s)\le 0$,
analytic on $a<\Re(s)<0$ and continuous on $a<\Re(s)\le 0$ (note for $I_+(s)$ this follows from the absolute convergence in $a<\Re(s)\le 0$).
It follows that $I_-(s)$ has an analytic continuation to the region $a<\Re(s)$.

Set
$$a_0:=\inf\{ s\in \R \ |\ I_-(s)<\infty\}.$$
Note that $a_0\le 0$.
We want to prove that $a_0\le a$. Assume that $a_0>a$. Consider the positive measure
$$\nu:=|f|^{a_0}\cdot \de_{|f|<1}\cdot \mu.$$
Then the integral
$$J(s):=\int_{Y(F)}|f|^s\cdot\nu$$
has the following properties:
\begin{itemize}
\item $J(s)<\infty$ for every $s>0$;
\item $J(s)$ has an analytic continuation to $a-a_0<\Re(s)$.
\end{itemize}
The analyticity of $J(s)$ near $s=0$ implies that $s$ near $0$ one has
$$J^{(n)}(s)=\int_{Y(F)}(\ln |f|)^n\cdot |f|^s\cdot\nu,$$
for each $n\ge 0$. Since $J(s)$ is analytic at $s=0$, there exists some $R>0$ and a constant $C>0$, such that
$$|J^{(n)}(0)|=\int_{Y(F)}(-\ln |f|)^n \nu\le C\cdot n!\cdot R^n.$$
But this implies that for $0\le t<R^{-1}$, one has
$$J(-t)=\int |f|^{-t} \nu=\int \sum_{n=0}^\infty \frac{t^n}{n!}(-\ln |f|)^n \nu\le \sum_{n=0}^\infty C (tR)^n<\infty.$$
Hence, $I_-(a_0-t)<\infty$ for $0\le t<R^{-1}$, contradicting the definition of $a_0$.
\end{proof}

We also need another type of extension of $\pi_{U,U_0,\kappa}$.

\begin{lemma}\label{extension-lem2}
Assume that the pair $(U_0,U)$ is $|\om|^{\kappa}$-nice.
Assume also that for every point $y\in (U_0/G)(F)$, there exists a neighborhood $V_y\sub (U_0/G)(F)$ of $y$ and
a closed subset $Z_y\sub X(F)$, such that $Z_y\sub U(F)$, and
for every $x\in U_0(F)$ mapping to $V_y$, one has $G(F)x\sub Z_y$.
Then the pair $(U_0,X)$ is $|\om|^{\kappa}$-nice.
\end{lemma}

\begin{proof} It is enough to prove the assertion about distributions $\pi_{X,U_0,\kappa}(\varphi)$ restricted to $V_y$, i.e.,
when $\psi$ is supported on $V_y$.
Set $W_y:=X(F)\setminus Z_y$.
Consider the open covering $X(F)=U(F)\cup W_y$. For every $\varphi\in\SS(X(F),|\om|^{\kappa})$, we can write
$$\varphi=\varphi_U+\varphi_{W_y},$$
where $\varphi_U\in \SS(U(F),|\om|^{\kappa})$ and $\varphi_{W_y}\in \SS(W_y,|\om|^{\kappa})$.
Now by our assumptions, the distribution
$$\pi_{X,U_0,\kappa}(\varphi_U)=\pi_{U,U_0,\kappa}(\varphi_U)$$
is well defined and of class $C^\infty$, while $\pi_{X,U_0,\kappa}(\varphi_{W_y})|_{V_y}\equiv 0$.
This immediately implies the assertion.
\end{proof}

\subsection{A version with smooth coarse moduli}


Let $X$ be a smooth $G$-scheme over $F$ (where $G=\GL_N$), $L$ a $G$-equivariant bundle on $X$, and $U_0\sub X$ an open $G$-invariant subset such that
there exists a smooth coarse moduli space $U_0/G$ of $[U_0/G]$. In other words, $U_0/G$ is a geometric quotient for the action of $G$ on $U_0$, but we do not
require the action of $G$ on $U_0$ to be free.
Then the natural projection $U_0\to U_0/G$ is smooth and its fibers are exactly $G$-orbits. Assume also that $L$ descends to a line bundle $\ov{L}$ on $U_0/G$.
We have natural integration maps 
$$\pi_{U_0,\kappa}:\SS(U_0(F),|L|^{\kappa})_{G(F)}\to \SS((U_0/G)(F),|\ov{L}|^{\kappa}),$$
and the results of Section \ref{integration-ext-sec} about extensions of these maps to maps $\pi_{X,U_0,\kappa}$ of the form
\eqref{lc-kappa-map}
are still valid.

\section{Application to $\Bun_G$}\label{Bun-sec}

\subsection{Filtration of the stack of rank $2$ bundles}\label{filtr-sec}

Let $C$ be a curve of genus $g\ge 2$.
Let $\Bun_{L_0}$ denote the stack of rank $2$ bundles $E$ on $C$, equipped with an isomorphism $\det(E)\simeq L_0$, where $L_0$ is a fixed line bundle.
Morphisms in this stack are isomorphisms compatible with the identifications of the determinant bundles. 
In particular, for $L_0=\OO$, $\Bun_{\OO}$ is the stack of $\SL_2$-bundles on $C$. 

It is well known that each stack $\Bun_{L_0}$ is admissible (but not of finite type).
We denote by $\Bun^{ss}_{L_0}\sub \Bun_{L_0}$ the substack of semistable bundles. 

\begin{definition}
For each $n$, we define an open substack $\Bun_{L_0}^{\le n}\sub \Bun_{L_0}$ consisting of bundles $E$ such that any line subbundle $A\sub E$ has $\deg(A)\le n$.
\end{definition}

The filtration $(\Bun_{\OO}^{\le n})$ plays an important role in the proof that the stack $\Bun_{\SL_2}$ is truncatable in \cite[Sec.\ 6]{DG}.
Note that $\Bun_{L_0}^{\le \lfloor \deg(L_0)/2 \rfloor}$ is exactly the substack of semistable bundles $\Bun^{ss}_{L_0}$.





\begin{definition}
For any $n$, we define $(\Bun_{L_0}^{\le n})^0\sub \Bun_{L_0}^{\le n}$ to be the substack consisting of $E$ such that for any degree $n$ line subbundle $A\sub E$ one has 
$H^1(A^2L_0^{-1})=0$. 
\end{definition}

\begin{lemma}\label{complement-bundle-lem}
(i) For $E\in \Bun_{L_0}^{\le n}$, we have $E\in (\Bun_{L_0}^{\le n})^0$ if and only if 
for any degree $n$ line bundle $A$ with $H^1(A^2L_0^{-1})\neq 0$ one has $\Hom(A,E)=0$. 

\noindent
(ii) $(\Bun_{L_0}^{\le n})^0$ is an open substack of $\Bun_{L_0}^{\le n}$.

\noindent
(iii) Every $E\in (\Bun_{L_0}^{\le n})^0\setminus\Bun_{L_0}^{\le n-1}$ has form $E=A\oplus A^{-1}L_0$, where $\deg(A)=n$ and
$H^1(A^2L_0^{-1})=0$.

\noindent
(iv) $(\Bun_{L_0}^{\le n})^0=\Bun_{L_0}^{\le n-1}$ unless $n\ge (g-1+\deg(L_0))/2$.

\noindent
(v) $(\Bun_{L_0}^{\le n})^0=\Bun_{L_0}^{\le n}$ for $n>g-1+\deg(L_0)/2$.
\end{lemma}

\begin{proof}
(i) We just have to note that if $A\to E$ is a nonzero morphism from a line bundle $A$ of degree $n$ then $A$ is necessarily is a subbundle of $E$ (otherwise, there would be a subbundle in $E$ of degree $>n$).

\noindent
(ii) The subvariety $Z\sub \Pic^n(C)$ of $A$ with $H^1(A^2L_0^{-1})\neq 0$ is closed in $\Pic^n(C)$, hence proper. Therefore, the locus of $E$ such that $\Hom(A,E)\neq 0$
for some $A\in Z$ is closed, and so its complement is open.

\noindent
(iii) For such $E$ there exists a line subbundle $A\sub E$ such that $\Ext^1(A^{-1}L_0,A)\simeq H^1(A^2L_0^{-1})=0$, so the extension $0\to A\to E\to E/A\to 0$, where
$E/A\simeq A^{-1}L_0$, splits.

\noindent
(iv) If $H^1(A^2L_0^{-1})=0$ then $\deg(A^2L_0^{-1})=2n-\deg(L_0)\ge g-1$. Hence, $E$ cannot have a line subbundle $A$ of degree $n$ with $H^1(A^2L_0^{-1})=0$
unless  $n\ge (g-1+\deg(L_0))/2$.

\noindent
(v) This follows from the fact that if $\deg(A^2L_0^{-1})=2n-\deg(L_0)>2g-2$ then $H^1(A^2L_0^{-1})=0$.
\end{proof}

 



\begin{lemma}\label{surjective-Bun-prop}
Assume $n\ge (g-1+\deg(L_0))/2$, and
$$\operatorname{Re}(\kappa)\cdot 2(2n-\deg(L_0))-(2n-\deg(L_0)-g+1)\not\in\Z_{\le 0}$$
(resp.,
$$\operatorname{Re}(\kappa)\neq \frac{2n-\deg(L_0)-g+1}{2(2n-\deg(L_0))}$$
in the non-archimedean case).
Then for every point $x\in ((\Bun_{L_0}^{\le n})^0\setminus \Bun_{L_0}^{\le n-1})(F)$, there exists a cocharacter
$\la^\vee:\G_m\to \Aut(x)$ such that in the archimedean case, for any character $a\mapsto a^m$ of $\G_m$ appearing in
$S^\bullet H^0T^*_{\Bun_{L_0}}|_x$, one has
$$\Re(\kappa)\cdot \lan \chi_{\LL,x},\la^\vee\ran+m\neq -\lan\De^{alg}_{\Aut(x)},\la^\vee\ran;$$
while in the non-archimedean case,
$$\Re(\kappa)\cdot \lan \chi_{\om,x},\la^\vee\ran\neq -\lan\De^{alg}_{\Aut(x)},\la^\vee\ran.$$
\end{lemma}

\begin{proof}
By Lemma \ref{complement-bundle-lem}(iii), every $E$ in 
$(\Bun_{L_0}^{\le n})^0\setminus \Bun_{L_0}^{\le n-1}$ has form
$E=A\oplus A^{-1}L_0$, with $\deg(A)=n$ and $H^1(A^2L_0^{-1})=0$. 
Note that this implies that $\deg(A^2L_0^{-1})\ge g-1>0$, hence $H^0(A^{-2}L_0)=0$.

One has
$$\und{\End}_0(E)\simeq \OO\oplus A^2L_0^{-1}\oplus A^{-2}L_0.$$
Hence, the Lie algebra of automorphisms of $E$ (with trivial determinant) is $H^0(\OO)\oplus H^0(A^2L_0^{-1})$,
and the fiber of the canonical line bundle of $\Bun_{L_0}$ at $E$ is identified with 
$$\om_{\Bun_{L_0}}|_E\simeq \det(H^1(\OO)^*)\ot\det(H^1(A^{-2}L_0)^*)\ot \det(H^0(A^2L_0)).$$

Let us consider the subgroup $\la^\vee:\G_m\hra \Aut(E)$ acting by $(a,a^{-1})$ on $A\oplus A^{-1}L_0$.
Since the induced action of $\G_m$ on $A^2L_0^{-1}$ (resp., $A^{-2}L_0$) has weight $2$ (resp., $-2$), we have
$$\lan\chi_{\om},\la^\vee\ran=2(\chi(C,A^2L_0^{-1})-\chi(C,A^{-2}L_0))=4(2n-\deg(L_0)),$$
$$-\lan\De^{alg}_{\Aut(E)},\la^\vee\ran=2 \chi(C,A^2L_0^{-1})=2(2n-\deg(L_0)-g+1)$$
(here we use the vanishing of $h^1(A^2L_0^{-1})$).

Also, the $\G_m$-weights of $H^0T^*_{\Bun_{L_0},E}\simeq H^1(A^{-2}L_0)^*\oplus H^1(\OO)^*$ are $0$ and $2$.
Hence, the weights on the symmetric algebra of $H^0T^*_{\Bun_{L_0},E}$ are non-negative even integers.
Now the assertion follows from our assumption on $\kappa$.
\end{proof}

\subsection{Limits of very stable bundles}\label{vs-lim-sec}




Recall that a rank $2$ bundle $E$ on $C$ is {\it very stable} if any nilpotent Higgs field $\phi:E\to E\ot\om_C$ is zero. It is known that a very stable bundle is stable.

We need some information on bundles which can appear as specializations of very stable bundles.
Recall that a bundle $E_0$ is a specialization of $E$ if there exists a family of bundles on $C$, $(E_s)$, parametrized by an irreducible base $S$ and a point $s_0\in S$
such that $E_s\simeq E$ for $s\neq s_0$ and $E_{s_0}\simeq E_0$.
 
\begin{lemma}\label{odd-degree-closure-lem}
Let $E$ be a very stable bundle with $\det(E)\simeq L_0$. 
Assume $E_0\in \Bun_{L_0}^{\le n}$, where $n\ge \deg(L_0)/2$, is such that $\Hom(E,E_0)\neq 0$.
Then $E_0\in (\Bun_{L_0}^{\le n})^0$.
In particular, $E$ does not specialize to any bundle in $\Bun_{L_0}^{\le n}\setminus (\Bun_{L_0}^{\le n})^0$, where
$n\ge\deg(L_0)/2$.
\end{lemma}

\begin{proof}
Assume that $\Hom(E,E_0)\neq 0$ and $E_0\in \Bun_{L_0}^{\le n}\setminus (\Bun_{L_0}^{\le n})^0$. 
Then there exists an exact sequence
$$0\to A\to E_0\to A^{-1}L_0\to 0$$
with $\deg(A)=n$, where $H^1(A^2L_0^{-1})\neq 0$. Since $n\ge\deg(L_0)/2$, we have
$\deg(A^{-1}L_0)\le \mu(E)=\deg(L_0)/2$,
hence $\Hom(E,A^{-1}L_0)=0$ by the stability of $E$.

Since $\Hom(E,E_0)\neq 0$, the exact sequence
$$0\to \Hom(E,A)\to \Hom(E,E_0)\to \Hom(E,A^{-1}L_0)=0$$
shows that $\Hom(E,A)\neq 0$.
Let $E\to A$ be a nonzero map. Such a map factors through a  surjection $E\to A(-D)$, 
where $D$ is an effective divisor.
Therefore, $E$ fits into an extension
$$0\to A^{-1}L_0(D)\to E\to A(-D)\to 0$$

Next, we observe that
$$\Hom(A,\om_CA^{-1}L_0)\simeq H^0(\om_CA^{-2}L_0)\simeq H^1(A^2L_0^{-1})^*\neq 0,$$ 
Hence, the composition
$$E\to A(-D)\to A\to \om_CA^{-1}L_0\to \om_CA^{-1}L_0(D)\to KE$$
gives a nonzero nilpotent Higgs field on $E$, so $E$ cannot be very stable. 
\end{proof}

\begin{cor}\label{specialization-cor} For $n<(g-1+\deg(L_0))/2$, a very stable bundle $E$ with determinant $L_0$ 
never specializes to a bundle in $\Bun_{L_0}^{\le n}$ which is not isomorphic to $E$.
\end{cor}

\begin{proof} 
Assume $E$ specializes to $E_0\in \Bun_{L_0}^{\le n}$, where $E_0\not\simeq E$. Applying Lemma \ref{odd-degree-closure-lem} iteratively,
and taking into account Lemma \ref{complement-bundle-lem}(iv), we see that $E_0$ is semistable.
But then the condition $\Hom(E,E_0)\neq 0$ gives a contradiction, since $E$ and $E_0$ have the same slope and $E$ is stable.
\end{proof}

The following result on possible specializations of very stable bundles is not needed for our argument but it clarifies the picture.

\begin{prop}
Let $E$ be a rank $2$ very stable bundle, and let $E_0$ be a specialization of $E$ such that $E_0\not\simeq E$.
Then $E\simeq A\oplus B$, where $A$ and $B$ are line bundles such that $\Ext^1(B,A)=0$.
\end{prop}

\begin{proof} Let $L_0:=\det(E)$. By Corollary \ref{specialization-cor}, $E_0\in \Bun_{L_0}^{\le n}\setminus \Bun_{L_0}^{\le n-1}$ for some $n\ge (g-1+\deg(L_0))/2$.
Hence, by Lemma \ref{odd-degree-closure-lem}, $E\in (\Bun_{L_0}^{\le n})^0\setminus \Bun_{L_0}^{\le n-1}$, and the assertion follows from 
Lemma \ref{complement-bundle-lem}(iii).
\end{proof}

\subsection{Main theorems}\label{main-thm-sec}


Consider the open substack
$\MM^{vs}_{L_0}\sub \Bun_{L_0}$ of very stable bundles, and let $M^{vs}_{L_0}$ be its coarse moduli space. The variety $M^{vs}_{L_0}$ is smooth and the line
bundle $\om$ on $\MM^{vs}_{L_0}$ descends to $\om$ on $M^{vs}_{L_0}$.

\begin{theorem}\label{main-thm-nice} 
Assume that $g\ge 2$. For every open substack of the form $[X/\GL_N]$ in $\Bun_{L_0}$, let $X^{vs}\sub X$ denote the open subset corresponding
to very stable bundles. Then for $\operatorname{Re}(\kappa)\ge 1/2$, the pair $(X^{vs},X)$ is $|\om|^{\kappa}$-nice, so 
there is a well defined map
$$\pi_\kappa:\SS(\Bun_{L_0},|\om|^{\kappa})\to C^\infty(M^{vs}_{L_0},|\om|^{\kappa}),$$
where $\pi_{\kappa}(\varphi)$, viewed as a distribution, is given by absolutely convergent integrals \eqref{distr-integral-eq}.
\end{theorem}

\begin{proof}
First, let $X^s\sub X$ be the $G$-invariant open subset corresponding to stable bundles (where $G=\GL_N$). Then we have a smooth geometric quotient $X^s/G$ of the action of $G$ on $X^s$,
which implies that the pair $(X^{vs},X^s)$ is $|\om|^{\kappa}$-nice for any $\kappa\in \C$.

Next, let us show that for $n<(g-1+\deg(L_0))/2$, the pair $(X^{vs},X^{\le n})$ is $|\om|^{\kappa}$-nice for any $\kappa\in \C$, where
$X^{\le n}\sub X$ is the $G$-invariant open subset corresponding to $\Bun_{L_0}^{\le n}$.
Indeed, since $(\Bun_{L_0}^{\le n})^0=\Bun_{L_0}^{\le n-1}$ for $n<(g-1+\deg(L_0))/2$ (see Lemma \ref{complement-bundle-lem}(iv)),
by Lemma \ref{odd-degree-closure-lem}, for such $n$ the $G$-orbits of points in $X^{vs}$ are contained in $X^s$.
Hence, the assertion follows from Lemma \ref{extension-lem2} applied to the embedding $X^s\hra X^{\le n}$.
In more detail, given a compact set $K$ of $G$-orbits in $X^{vs}/G$,
we have a closed subset $Z_K\sub X(F)$ consisting of points corresponding to bundles $E_0$ such that $\Hom(E,E_0)\neq 0$ for some $E\in K$.
By Lemma \ref{odd-degree-closure-lem}, we have $Z_K\sub X^s(F)$. On the other hand, by semicontinuity of $\dim\Hom(E,E_0)$,
the closures of $G$-orbits in $K$ are contained in $Z_K$. Hence, we can apply Lemma \ref{odd-degree-closure-lem} to deduce that the pair $(X^{vs},X^{\le n})$
is $|\om|^{\kappa}$-nice.

Next, we will use induction on $n$ to prove that for every $n$, the pair $(X^{vs},X^{\le n})$ is $|\om|^{\kappa}$-nice for $\Re(\kappa)\ge 1/2$.
For some $n\ge (g-1+\deg(L_0))/2$, let 
$$X^{\le n-1}\sub (X^{\le n})^0\sub X$$ 
denote the open subsets corresponding to the substacks $\Bun_{L_0}^{\le n-1}\sub (\Bun_{L_0}^{\le n})^0$.
By the induction assumption (or the case $n<(g-1+\deg(L_0))/2$), we can assume that the claim holds for $X^{\le n-1}$, i.e, the pair $(X^{vs},X^{\le n-1})$
is $|\om|^{\kappa}$-nice.

Note that we have
$$\operatorname{Re}(\kappa)\ge \frac{1}{2}>\frac{2n-\deg(L_0)-g+1}{2(2n-\deg(L_0))}.$$
Hence, by Lemma \ref{surjective-Bun-prop}, we can apply by Lemma \ref{extension-lem1}
to the inclusion
$[X^{\le n-1}/G]\sub [(X^{\le n})^0/G]$ and to the interval 
$$(a,b)=(\frac{2n-\deg(L_0)-g+1}{2(2n-\deg(L_0))},b)$$
with $b$ sufficiently large,
and deduce that the pair $(X^{vs},(X^{\le n})^0)$ is $|\om|^{\kappa}$-nice. 

Finally, we would like to apply Lemma \ref{extension-lem2} to the embedding $(X^{\le n})^0\sub X^{\le n}$.
We need to check that the assumptions of this lemma are satisfied. Given a compact set $K$ of $G$-orbits in $X^{vs}/G$,
we have a closed subset $Z_K\sub X(F)$ consisting of points corresponding to bundles $E_0$ such that $\Hom(E,E_0)\neq 0$ for some $E\in K$.
By Lemma \ref{odd-degree-closure-lem}, we have $Z_K\sub (X^{\le n})^0(F)$. On the other hand, by semicontinuity of $\dim\Hom(E,E_0)$,
the closures of $G$-orbits in $K$ are contained in $Z_K$. Hence, by Lemma \ref{extension-lem2}, the pair $(X^{vs},X^{\le n})$ is $|\om|^{\kappa}$-nice.

Since $X^{\le n}=X$ for sufficiently large $n$, the assertion follows.
\end{proof}

In the non-archimedean case we can also prove the following boundedness result, which implies that the stack $\SS(\Bun_{L_0})$ is $\kappa$-bounded for each $\kappa\in\C$
in the terminology of \cite[Def.\ 2.10]{BK}.

\begin{theorem}\label{main-thm-bound}
Assume $F$ is non-archimedean and $g\ge 2$. 

\noindent
(i) For $n>g-2+\deg(L_0)/2$ and $\Re(\kappa)\ge 1/2$, the natural map
$$\SS(\Bun_{L_0}^{\le n},|\om|^{\kappa})\to \SS(\Bun_{L_0},|\om|^{\kappa})$$
is an isomorphism. 

\noindent
(ii) For any $\kappa\in\C$, the map
$$\SS(\Bun_{L_0}^{\le n},|\om|^{\kappa})\to \SS(\Bun_{L_0},|\om|^{\kappa})$$
is an isomorphism for large enough $n$.
\end{theorem}

\begin{proof}
(i) It suffices to prove that under these assumptions the map
\begin{equation}\label{Bun-n-n+1-kappa-map}
\SS(\Bun_{L_0}^{\le n},|\om|^{\kappa})\to \SS(\Bun_{L_0}^{\le n+1},|\om|^{\kappa})
\end{equation}
is an isomorphism. We are going to apply Lemma \ref{surjective-lem-nonarch}(i) to the inclusion $\Bun_{L_0}^{\le n}\hra \Bun_{L_0}^{\le n+1}$.

First of all, we observe that by Lemma \ref{complement-bundle-lem}(v), we have
$\Bun_{L_0}^{\le n+1}=(\Bun_{L_0}^{\le n+1})^0$, so by Lemma \ref{complement-bundle-lem}(iii), every
$E\in \Bun_{L_0}^{\le n+1}\setminus \Bun_{L_0}^{\le n}$ has form $E=A\oplus A^{-1}L_0$, where $H^1(A^2L_0^{-1})=0$. In this case $\deg(A^2L_0^{-1})>0$,
so $H^0(A^{-2}L_0)=0$. It follows that the group $\Aut(E)$ is the semidirect product of $\G_m$ acting by $(a,a^{-1})$ on $A\oplus A^{-1}L_0$ with the additive group
$\Hom(A^{-1}L_0,A)=H^0(A^2L_0^{-1})$. Together with Lemma \ref{surjective-Bun-prop} this implies that the assumptions of Lemma \ref{surjective-lem-nonarch}(i) are
satisfied and we deduce that the map \eqref{Bun-n-n+1-kappa-map} is an isomorphism.

\noindent
(ii) Pick $N$ such that $N>g-2+\deg(L_0)/2$ and $\Re(\kappa)\neq \frac{2n-\deg(L_0)-g+1}{2(2n-\deg(L_0))}$
for $n>N$. Then the same argument as in (i) shows that \eqref{Bun-n-n+1-kappa-map} is an isomorphism for $n\ge N$.
\end{proof}

Note that Theorems \ref{main-thm-nice} and \ref{main-thm-bound} imply \cite[Conj.\ 3.5(1)]{BK} for $G=\SL_2$.

\end{document}